\documentclass{amsart}
\usepackage{amsmath,amsthm,amssymb}
\usepackage{enumerate}
\usepackage{color}

\newtheorem{thm}{Theorem}[section]
\newtheorem{cor}[thm]{Corollary}
\newtheorem{lemma}[thm]{Lemma}

\theoremstyle{definition}

\numberwithin{equation}{section}

\newtheorem{prop}[thm]{Proposition}

\newcommand{\dbarb}{\overline{\partial}_b}

\newcommand{\zbar}{\overline{z}}
\newcommand{\Zbar}{\overline{Z}}

\newcommand{\tomega}{\tilde{\omega}}

\begin{document}

%%%%% To ease editing, add:

%\baselineskip=17pt

%%%%%%%%%%%%%%%%

%% In the running head, give an abbreviation of the title.
%\titlerunning{A subelliptic Bourgain-Brezis inequality}

\title{A subelliptic Bourgain-Brezis inequality}

\author{Yi Wang}
\address{Department of Mathematics, Stanford University, Stanford, CA 94305}
\email{wangyi@math.stanford.edu}
\author{Po-Lam Yung}
\address{Department of Mathematics, Rutgers, the State University of New Jersey, Piscataway, NJ 08854}
\email{pyung@math.rutgers.edu}
\date{}

%\subjclass{Primary 35H20; Secondary 42B35}

%%%%%%%%

\begin{abstract}
We prove an approximation lemma on (stratified) homogeneous groups that allows one to approximate a function in the non-isotropic Sobolev space $\dot{NL}^{1,Q}$ by $L^{\infty}$ functions, generalizing a result of Bourgain-Brezis \cite{MR2293957}. We then use this to obtain a Gagliardo-Nirenberg inequality for $\dbarb$ on the Heisenberg group $\mathbb{H}^n$.

%% Keywords are optional
%\keywords{div-curl, compensation phenomena, critical Sobolev embedding, homogeneous groups}
\end{abstract}

\maketitle

\section{Introduction} \label{sect:intro}

In this paper, we study some subelliptic compensation phenomena on homogeneous groups, that have to do with divergence, curl and the space $L^1$ of Lebesgue integrable functions or differential forms. In the elliptic cases they were discovered by Bourgain-Brezis, Lanzani-Stein and van Schaftingen around 2004. Also lying beneath our results is the failure of the critical Sobolev embedding of the non-isotropic Sobolev space $\dot{NL}^{1,Q}$ into $L^{\infty}$. In particular, we prove an approximation lemma that describes how functions in $\dot{NL}^{1,Q}$ can be approximated by functions in $L^{\infty}$.

To begin with, let us describe the elliptic results on $\mathbb{R}^n$ ($n \geq 2$) upon which our results are based. We denote by $d$ the Hodge-de Rham exterior derivative, and $d^*$ its (formal) adjoint. The theory discovered by Bourgain-Brezis, Lanzani-Stein and van Schaftingen consists of three major pillars, each best illustrated by a separate theorem. The first involves the solution of $d^*$:

\begin{thm}[Bourgain-Brezis \cite{MR2293957}] \label{thm:BB}
Suppose $q \ne n-1$. Then for any $q$-form $f$ with coefficients on $L^n(\mathbb{R}^n)$ that is in the image\footnote{By this we mean $f$ is the $d^*$ of some form with coefficients in $\dot{W}^{1,n}(\mathbb{R}^n)$, where $\dot{W}^{1,n}(\mathbb{R}^n)$ is the (homogeneous) Sobolev space of functions that have 1 derivative in $L^n(\mathbb{R}^n)$.} of $d^*$, there exists a $(q+1)$-form $Y$ with coefficients in $L^{\infty}(\mathbb{R}^n)$ such that $$d^*Y = f$$ in the sense of distributions, and $\|Y\|_{L^{\infty}(\mathbb{R}^n)} \leq C \|f\|_{L^n(\mathbb{R}^n)}.$
\end{thm}

In particular, we have
\begin{cor}[Bourgain-Brezis \cite{MR1949165}] \label{cor:BB}
For any function $f \in L^n(\mathbb{R}^n)$, there exists a vector field $Y$ with coefficients in $L^{\infty}(\mathbb{R}^n)$ such that $$\mathrm{div\,} Y = f$$ in the sense of distributions, and $\|Y\|_{L^{\infty}(\mathbb{R}^n)} \leq C \|f\|_{L^n(\mathbb{R}^n)}.$
\end{cor}

The second pillar is a Gagliardo-Nirenberg inequality for differential forms:

\begin{thm}[Lanzani-Stein \cite{MR2122730}] \label{thm:LS}
Suppose $u$ is a $q$-form on $\mathbb{R}^n$ that is smooth with compact support. We have $$\|u\|_{L^{n/(n-1)}(\mathbb{R}^n)} \leq C (\|du\|_{L^1(\mathbb{R}^n)} + \|d^*u\|_{L^1(\mathbb{R}^n)})$$ unless $d^*u$ is a function or $du$ is a top form. If $d^*u$ is a function, one needs to assume $d^*u = 0$; if $du$ is a top form, one needs to assume $du = 0$. Then the above inequality remains true.
\end{thm}

Since $d$ of a 1-form is its curl and $d^*$ of a 1-form is its divergence, this is sometimes called a div-curl inequality.

The third theorem is the following compensation phenomenon:

\begin{thm}[van Schaftingen \cite{MR2078071}] \label{thm:vs}
If $u$ is a $C^{\infty}_c$ 1-form on $\mathbb{R}^n$ with $d^*u = 0$, then for any 1-form $\phi$ with coefficients in $C^{\infty}_c(\mathbb{R}^n)$, we have
$$\int_{\mathbb{R}^n} u \cdot \phi dx \leq C \|u\|_{L^1(\mathbb{R}^n)} \|\phi\|_{\dot{W}^{1,n}(\mathbb{R}^n)}.$$
\end{thm}

If $\dot{W}^{1,n}(\mathbb{R}^n)$ were embedded into $L^{\infty}(\mathbb{R}^n)$, Theorem \ref{thm:BB} would be trivial by Hodge decomposition, and so would be Theorem \ref{thm:vs} by H\"{o}lder's inequality. It is remarkable that these theorems remain to hold even though the desired Sobolev embedding fails.

It turns out all three theorems above are equivalent by duality. van Schaftingen \cite{MR2078071} gave a beautiful elementary proof of Theorem \ref{thm:vs}, thereby proving all of them.

We mention here that these results seem to be quite different from the more classical theory of compensated compactness; no connection between them is known so far.

We also refer the reader to the work of Brezis-van Schaftingen \cite{MR2332419}, Chanillo-van Schaftingen~\cite{MR2511628}, Maz'ya \cite{MR2578609}, Mironescu \cite{MR2645163}, Mitrea-Mitrea \cite{MR2470831}, van Schaftingen \cite{MR2443922}, \cite{MR2550188} and Amrouche-Nguyen \cite{MR2771258} for some interesting results related to these three theorems. In particular, Chanillo-van Schaftingen proved in \cite{MR2511628} a generalization of Theorem~\ref{thm:vs} to general homogeneous groups.

On the other hand, in \cite{MR2293957}, Bourgain-Brezis proved the following remarkable theorem, strengthening all three theorems above:
\begin{thm}[Bourgain-Brezis \cite{MR2293957}]
In Theorem~\ref{thm:BB} and Corollary~\ref{cor:BB}, the space $L^{\infty}(\mathbb{R}^n)$ can be replaced by the smaller Banach space $L^{\infty}(\mathbb{R}^n) \cap \dot{W}^{1,n}(\mathbb{R}^n)$. Also, in Theorem~\ref{thm:LS} and~\ref{thm:vs}, the spaces $L^1(\mathbb{R}^n)$ can be replaced by the bigger Banach space $L^1(\mathbb{R}^n) + (\dot{W}^{1,n}(\mathbb{R}^n))^*$. (Here $X^*$ denotes the dual of a Banach space $X$.)
\end{thm}
They proved this by giving a direct constructive proof of the analog of Theorem~\ref{thm:BB}, where the space $L^{\infty}(\mathbb{R}^n)$ is replaced by $L^{\infty}(\mathbb{R}^n) \cap \dot{W}^{1,n}(\mathbb{R}^n)$; they then deduced the rest by duality. In the former they used the following approximation lemma, which is another remedy of the failure of the critical Sobolev embedding, and which is of independent interest:

\begin{lemma}[Bourgain-Brezis \cite{MR2293957}] \label{lem:approx}
Given any $\delta > 0$ and any function $f \in \dot{W}^{1,n}(\mathbb{R}^n)$, there exist a function $F \in L^{\infty}(\mathbb{R}^n) \cap \dot{W}^{1,n}(\mathbb{R}^n)$ and a constant $C_{\delta} > 0$, with $C_{\delta}$ independent of $f$, such that $$\sum_{i=2}^n \|\partial_i f - \partial_i F\|_{L^n(\mathbb{R}^n)} \leq \delta \|\nabla f\|_{L^n(\mathbb{R}^n)}$$ and $$\|F\|_{L^{\infty}(\mathbb{R}^n)} + \|\nabla F\|_{L^n(\mathbb{R}^n)} \leq C_{\delta} \|\nabla f\|_{L^n(\mathbb{R}^n)}.$$
\end{lemma}

Here one should think of $F$ as an $L^{\infty}(\mathbb{R}^n) \cap \dot{W}^{1,n}(\mathbb{R}^n)$ function whose derivatives approximate those of the given $f$ in all but one direction.

In this paper, we prove an analog of the above approximation lemma on any homogeneous group $G$. To describe our result we need some notations. First, let $\mathfrak{g}$ be a Lie algebra (over $\mathbb{R}$) that is \emph{graded}, in the sense that $\mathfrak{g}$ admits a decomposition $$\mathfrak{g} = V_1 \oplus V_2 \oplus \dots \oplus V_m$$ into direct sums of subspaces $V_1, \dots, V_m$ of $\mathfrak{g}$ such that $$[V_{j_1}, V_{j_2}] \subseteq V_{j_1+j_2}$$ for all $j_1$, $j_2$, where $V_j$ is understood to be zero if $j > m$. We assume that $V_m \ne \{0\}$. It is immediate that $\mathfrak{g}$ is nilpotent of step $m$. We introduce a natural family of dilations on $\mathfrak{g}$, by letting $$\lambda \cdot v = \lambda v_1 + \lambda^2 v_2 + \dots + \lambda^m v_m$$ if $v = v_1 + \dots + v_m$, $v_i \in V_i$ and $\lambda > 0$. This defines a one-parameter family of algebra automorphisms of $\mathfrak{g}$. Furthermore, we assume that $\mathfrak{g}$ is \emph{stratified}, in the sense that $V_1$ generates $\mathfrak{g}$ as a Lie algebra. Let $G$ be the connected and simply connected Lie group whose Lie algebra is $\mathfrak{g}$. Such a Lie group $G$ with stratified $\mathfrak{g}$ is then called a \emph{homogeneous group}. It carries a one-parameter family of automorphic dilations, given by $\lambda \cdot \exp(v) := \exp(\lambda \cdot v)$ where $\exp \colon \mathfrak{g} \to G$ is the exponential map. In the sequel we fix such a group $G$.

Now define the homogeneous dimension $Q$ of $G$ by $$Q := \sum_{j=1}^m j \cdot n_j$$ where $n_j := \dim V_j$. We also pick a basis $X_1$, $\dots$, $X_{n_1}$ of $V_1$. Any linear combination of these will then be a left-invariant vector field of degree 1 on $G$. If $f$ is a function on $G$, we define its subelliptic gradient as the $n_1$-tuple $$\nabla_b f := (X_1 f, \dots, X_{n_1} f).$$ The homogeneous non-isotropic Sobolev space $\dot{NL}^{1,Q}(G)$ is then the space of functions on $G$ whose subelliptic gradient is in $L^Q(G)$. Here in defining the $L^Q(G)$ space, we use the Lebesgue measure on $\mathfrak{g}$, which we identify with $G$ via the exponential map. In the following, we will denote the functional spaces on $G$ by $\dot{NL}^{1,Q}$, $L^Q$, $L^{\infty}$ etc. for simplicity unless otherwise specified.

It is well-known that $\dot{NL}^{1,Q}$ fails to embed into $L^{\infty}$. Nonetheless, we prove the following approximation lemma for functions in $\dot{NL}^{1,Q}$:

\begin{lemma} \label{lem:approxsub}
Given any $\delta > 0$ and any function $f$ on $G$ with $\|\nabla_b f\|_{L^Q} < \infty$, there exist a function $F \in L^{\infty}$ with $\nabla_b F \in L^Q$, and a constant $C_{\delta} > 0$ with $C_{\delta}$ independent of $f$, such that $$\sum_{k=2}^{n_1} \|X_k f - X_k F\|_{L^Q} \leq \delta \|\nabla_b f\|_{L^Q}$$ and $$\|F\|_{L^{\infty}} + \|\nabla_b F\|_{L^Q} \leq C_{\delta} \|\nabla_b f\|_{L^Q}.$$
\end{lemma}

Specializing this result to the Heisenberg group $\mathbb{H}^n$, we deduce, for instance, the following result about the solution of $\dbarb$:

\begin{thm} \label{thm:soldbarbstrong}
Suppose $Q = 2n+2$ and $q \ne n-1$. Then for any $(0,q)$-form $f$ on $\mathbb{H}^n$ that has coefficients in $L^Q$ and that is the $\dbarb^*$ of some other form with coefficients in $\dot{NL}^{1,Q}$, there exists a $(0,q+1)$-form $Y$ on $\mathbb{H}^n$ with coefficients in $L^{\infty} \cap \dot{NL}^{1,Q}$ such that $$\dbarb^*Y = f$$ in the sense of distributions, with $\|Y\|_{L^{\infty}} + \|\nabla_b Y\|_{L^Q} \leq C \|f\|_{L^Q}$.
\end{thm}

We then have a Gagliardo-Nirenberg inequality for $\dbarb$ on $\mathbb{H}^n$:

\begin{thm} \label{thm:subGNstrong}
Suppose $Q = 2n+2$. If $u$ is a $(0,q)$ form on $\mathbb{H}^n$ with $2 \leq q \leq n-2$, then
\begin{equation} \label{eq:GNdbarbq}
\|u\|_{L^{Q/(Q-1)}} \leq C (\|\dbarb u\|_{L^1 + (\dot{NL}^{1,Q})^*} + \|\dbarb^* u \|_{L^1 + (\dot{NL}^{1,Q})^*}).
\end{equation}
Also, if $n \geq 2$ and $u$ is a function on $\mathbb{H}^n$ that is orthogonal to the kernel of $\dbarb$, then
\begin{equation} \label{eq:GNdbarb0}
\|u\|_{L^{Q/(Q-1)}} \leq C \|\dbarb u\|_{L^1 + (\dot{NL}^{1,Q})^*}.
\end{equation}
\end{thm}

There is also a version of this result for (0,1) forms and $(0,n-1)$ forms, analogous to the last part of Theorem~\ref{thm:LS}.

A weaker version of this theorem, namely what one has by replacing $L^1+(\dot{NL}^{1,Q})^*$ above by $L^1$, can also be deduced easily from the work of Chanillo-van Schaftingen \cite{MR2511628} (c.f. also \cite{MR2592736}).

Several difficulties need to be overcome when we prove Lemma~\ref{lem:approxsub} on a general homogeneous group. The first is that we no longer have a Fej\'{e}r kernel as in the Euclidean spaces, which served as the building block of a good reproducing kernel $K_j$ in the original proof of Bourgain-Brezis. As a result, we need to find an appropriate variant of that. What we do is to adopt the heat kernels $S_j$, and to use $S_{j+N}$, where $N$ is large, as our approximate reproducing kernel. In other words, we use $S_{j+N} \Delta_j f$, where $N$ is large, to approximate $\Delta_j f$, where $\Delta_j f$ is a Littlewood-Paley piece of the function $f$. Since the heat kernel does not localize perfectly in ``frequency'', we need, in the preparational stage, some extra efforts to deal with additional errors that come up in that connection.

Our second difficulty, which is also the biggest challenge, is that our homogeneous group is in general not abelian. Hence we must carefully distinguish between left- and right-invariant derivatives when we differentiate a convolution (which is defined in (\ref{eq:convdef})): $X_k (f*K)$ is equal to $f*(X_k K)$, and not to $(X_k f)*K$, if $X_k$ is left-invariant and $K$ is any kernel (c.f Proposition~\ref{prop:deriv} in Section~\ref{sect:leftright}). To get around that, several ingredients are involved. One of them is to explore the relationship between left- and right-invariant vector fields, which we recall in Section \ref{sect:leftright}. Another is to introduce two different auxiliary controlling functions $\omega_j$ and $\tomega_j$. These are functions that dominate $|\Delta_j f|$ pointwisely (at least morally), and both $X_k \omega_j$ and $X_k \tomega_j$, for $k=2,\dots,n_1$, will be better controlled than $X_1 \omega_j$ and $X_1 \tomega_j$. The key here, on the other hand, is three-fold: first, $\tomega_j$ is frequency localized; second, $\tomega_j$ dominates $\omega_j$; finally, one has a good bound on $$\|\sup_j (2^j \omega_j)\|_{L^Q},$$ as we will see in Proposition~\ref{prop:maxomega}. On the contrary, we wish to point out that $\tomega_j$ will not satisfy the analog of Proposition~\ref{prop:maxomega}, and $\omega_j$ will not be frequency localized. This is basically why we needed to construct both auxiliary functions $\omega_j$ and $\tomega_j$. In defining such $\omega_j$ and $\tomega_j$, instead of taking an ``$L^{\infty}$ convolution'' as in the definition of $\omega_j$ used by Bourgain-Brezis, we will take a discrete convolution in $l^Q$, and an honest convolution, for $\omega_j$ and $\tomega_j$ respectively. (The precise definition of $\omega_j$ and $\tomega_j$ can be found in Section \ref{sect:outline}.) We then use $\omega_j$ to control the part of $f$ where the high frequencies are dominating, and use $\tomega_j$ to control the other part of $f$ where the low frequencies are dominating.

Finally, we will need two slightly different versions of Littlewood-Paley theories on a homogeneous group. One is chosen such that $f = \sum_j \Delta_j f$, and the other is chosen such that the reverse Littlewood-Paley inequality holds (as in Proposition~\ref{prop:equivLP}).

We will now proceed as follows. In Section~\ref{sect:pre}-\ref{sect:LP} we describe some preliminaries about homogeneous groups. This includes some mean-value type inequalities on $G$, some tools that allow us to mediate between left- and right-invariant derivatives, as well as a refinement of a Littlewood-Paley theory on $G$. In Section~\ref{sect:alg} we give some algebraic preliminaries needed in the proof of Lemma~\ref{lem:approxsub}, and in Section~\ref{sect:outline} we give an outline of the proof of Lemma~\ref{lem:approxsub}. Section \ref{sect:f0}-\ref{sect:g0} contains the details of the proof of Lemma \ref{lem:approxsub}. Finally in Section \ref{sect:thms} we prove Theorem \ref{thm:soldbarbstrong} and \ref{thm:subGNstrong}.

\section{Preliminaries} \label{sect:pre}

Let $G$ be a homogeneous group, $n_j := \dim V_j$, and $X_1, \dots, X_{n_1}$ be a basis of $V_1$ as above. We introduce a coordinate system on $G$. First write $$n := n_1 + \dots + n_m,$$ and extend $X_1, \dots, X_{n_1}$ to a basis $X_1, \dots, X_n$ of $\mathfrak{g}$, such that $X_{n_{j-1} + 1}, \dots, X_{n_j}$ is a basis of $V_j$ for all $1 \leq j \leq m$ (with $n_0$ understood to be 0). Then for $x = [x_1, \dots, x_n] \in \mathbb{R}^n$, we identify $x$ with $\sum_{i=1}^n x_i X_i \in \mathfrak{g}$. We will also identify $\mathfrak{g}$ with $G$ via the exponential map. Thus we write $x$ for the point $\exp(\sum_{i=1}^n x_i X_i) \in G$. This defines a coordinate system on $G$. The group identity of $G$ is $0=[0,\dots,0]$, and the dilation on $G$ is given explicitly by $$\lambda \cdot x = [\lambda x_1, \dots, \lambda x_{n_1}, \lambda^2 x_{n_1+1}, \dots, \lambda^2 x_{n_2}, \dots, \lambda^m x_{n_{m-1}+1}, \dots, \lambda^m x_n]$$ for $\lambda > 0$ and $x = [x_1, \dots, x_n]$.

For $x, y \in G$, we write $x \cdot y$ for their group product in $G$. By the Campbell-Hausdorff formula, this group law is given by a  polynomial map when viewed as a map from $\mathbb{R}^n \times \mathbb{R}^n \to \mathbb{R}^n$. More precisely, the map $(x,y) \mapsto x \cdot y$ can be computed by
\begin{align}
x \cdot y &= \exp\left(\sum_{i=1}^n x_i X_i \right) \cdot \exp\left(\sum_{i=1}^n y_i X_i \right) \notag \\
&= \exp\left( \sum_{i=1}^n x_i X_i + \sum_{i=1}^n y_i X_i + \frac{1}{2} \left[\sum_{i=1}^n x_i X_i, \sum_{i=1}^n y_i X_i\right] + \dots \right). \label{eq:CH}
\end{align}
It follows that for $1 \leq k \leq n_1$, the $k$-th coordinate of $x \cdot y$ is $x_k + y_k$.

The dilations on $G$ are automorphisms of the group: in particular,
\begin{equation} \label{eq:dilauto}
\lambda \cdot (x \cdot y) = (\lambda \cdot x) \cdot (\lambda \cdot y)
\end{equation}
for all $\lambda > 0$ and all $x, y \in G$.

A function $f(x)$ on $G$ is said to be homogeneous of degree $l$ if $f(\lambda \cdot x) = \lambda^l f(x)$ for all $x \in G$ and $\lambda > 0$. From (\ref{eq:CH}) we see that for all $n_j < k \leq n_{j+1}$, the $k$-th coordinate of $x \cdot y$ is equal to $x_k + y_k + P_k(x,y)$  where $P_k$ is a homogeneous polynomial of degree $j$ on $G \times G$. On $G$ one can define the homogeneous norm $$\|x\| = \left(\sum_{j=1}^m  \sum_{n_{j-1} < k  \leq n_j} |x_k|^{\frac{2m!}{j}} \right)^{\frac{1}{2m!}}.$$ It is a homogeneous function of degree 1 on $G$, and satisfies a quasi-triangle inequality $$\|x \cdot y\| \leq C(\|x\| + \|y\|)$$ for all $x, y \in G$, where $C$ is a constant depending only on $G$. We also have $\|x\|=\|x^{-1}\|$ for all $x \in G$, since if $x = [x_1, \dots, x_n]$ then $x^{-1} = [-x_1, \dots, -x_n]$.

Any element $X$ of $\mathfrak{g}$ can be identified with a left-invariant vector field on $G$. It will be said to be homogeneous of degree $l$ if $X(f(\lambda \cdot x)) = \lambda^l (Xf)(\lambda \cdot x)$ for all $C^1$ functions~$f$. $X_1, \dots, X_{n_1}$ is then a basis of left-invariant vector fields of degree 1 on $G$. We remind the reader that we write $\nabla_b f = (X_1 f, \dots, X_{n_1} f)$, and call this the subelliptic gradient of $f$.

By the form of the group law on $G$, one can see that if $n_{j-1} < k \leq n_j$, then $X_k$ can be written as 
\begin{equation} \label{eq:repleftinvder}
X_k = \frac{\partial}{\partial x_k} + \sum_{p = j+1}^m \sum_{n_{p-1} < k' \leq n_p} P_{k,k'}(x) \frac{\partial}{\partial x_{k'}}
\end{equation}
where $P_{k,k'}(x)$ is a homogeneous polynomial of degree $p-j$ if $n_{p-1} < k' \leq n_p$.

If $X$ is a left-invariant vector field on $G$, we write $X^R$ for the right-invariant vector field on $G$ that agrees with $X$ at the identity (namely $0$). We also write $\nabla_b^R f$ for the $n_1$-tuple $(X_1^R f, \dots, X_{n_1}^R f)$.

The Lebesgue measure $dx$ on $\mathbb{R}^n$ is a Haar measure on $G$ if we identify $x \in \mathbb{R}^n$ with a point in $G$ as we have always done. It satisfies $d(\lambda \cdot x) = \lambda^Q dx$ for all positive $\lambda$, where $Q = \sum_{j=1}^m j \cdot n_j$ is the homogeneous dimension we introduced previously.

With the Haar measure we define the $L^p$ spaces on $G$. If $f$ and $g$ are two $L^1$ functions on $G$, then their convolution is given by
\begin{equation} \label{eq:convdef}
f*g(x) = \int_G f(x \cdot y^{-1}) g(y) dy,
\end{equation}
or equivalently
$$ f*g(x) = \int_G f(y) g(y^{-1} \cdot x) dy. $$
The non-isotropic Sobolev space $\dot{NL}^{1,Q}$ is the space of functions $f$ such that $\|\nabla_b f\|_{L^Q} < \infty$.

\section{Some basic inequalities}

To proceed further, we collect some basic inequalities that will be useful on a number of occasions.

In this and the next section, $f$ and $g$ will denote two general $C^1$ functions on $G$. The first proposition is a mean-value inequality.

\begin{prop} \label{prop:meanvalue}
There exist absolute constants $C > 0$ and $a > 0$ such that $$|f(x \cdot y^{-1}) - f(x)| \leq C \|y\| \sup_{\|z\| \leq a \|y\|} |(\nabla_b f)(x \cdot z^{-1})|$$ for all $x, y \in G$.
\end{prop}

For a proof of this proposition, see Folland-Stein \cite[Page 33, (1.41)]{MR657581}.

There is also a mean-value inequality for right translations, whose proof is similar and we omit:

\begin{prop} \label{prop:meanvalueright}
There exist absolute constants $C > 0$ and $a > 0$ such that $$|f(y^{-1} \cdot x) - f(x)| \leq C \|y\| \sup_{\|z\| \leq a\|y\|} |(\nabla_b^R f)(z^{-1} \cdot x)|$$ for all $x, y \in G$.
\end{prop}

Next we have some integral estimates:

\begin{prop} \label{prop:fg}
If
\begin{equation} \label{eq:condfdiffintgk}
\left|\nabla_b f(x) \right| \leq (1+\|x\|)^{-M} \quad \text{and} \quad \left|g(x)\right| \leq (1+\|x\|)^{-M}
\end{equation}
for all $M > 0$, then for any non-negative integer $k$,
\begin{equation} \label{eq:fdiffintgk}
\int_G |f(x \cdot y^{-1}) - f(x)| |g_k(y)| dy \leq C 2^{-k} (1+\|x\|)^{-M}
\end{equation}
for all $M$, where $g_k(y) := 2^{kQ} g(2^k \cdot y)$.
\end{prop}

The key here, as well as in the next two propositions, is that we get a small factor $2^{-k}$ on the right hand side of our estimates.

\begin{proof}
We split the integral into two parts: $$\int_G |f(x \cdot y^{-1}) - f(x)| |g_k(y)| dy =  \int_{\|y\| \leq c\|x\|} dy + \int_{\|y\| > c \|x\|} dy = I + II,$$ where $c$ is a constant chosen such that if $\|y\| \leq c\|x\|$ and $\|z\| \leq a\|y\|$ then $\|x \cdot z^{-1}\| \geq \frac{1}{2} \|x\|$. Here $a$ is the constant appearing in the statement of Proposition~\ref{prop:meanvalue}, and such $c$ exists by  (\ref{eq:normmeanvalue}) below. We then apply Proposition \ref{prop:meanvalue} twice. First in $I$, the integrand can be bounded by $$|f(x \cdot y^{-1}) - f(x)||g_k(y)|  \leq C (1+\|x\|)^{-M} \|y\| |g_k(y)|$$ for all $M$, and $$\int_G \|y\| |g_k(y)| dy = C 2^{-k}.$$ Also, in $II$, the integrand can be bounded by $$|f(x \cdot y^{-1}) - f(x)||g_k(y)| \leq C \|y\| |g_k(y)|,$$ and \begin{align*} \int_{\|y\| \geq c\|x\|} \|y\| |g_k(y)| dy &\leq \int_{\|y\| \geq c\|x\|} \|y\| (1+2^k\|y\|)^{-Q-M-1} 2^{kQ} dy \\ &\leq C 2^{-k} (1+\|x\|)^{-M} \end{align*} for all $M$. Combining the estimates concludes the proof.
\end{proof}

We remark here that if we want (\ref{eq:fdiffintgk}) to hold for a specific $M$, then we only need condition (\ref{eq:condfdiffintgk}) to hold with $M$ replaced by $Q+M+1$.

In particular, we have
\begin{prop} \label{prop:fgconv}
If $f$, $g$ are as in Proposition \ref{prop:fg}, and in addition $\int_G g(y) dy = 0$, then for any non-negative integer $k$, we have $$|f*g_k(x)| \leq C 2^{-k} (1+\|x\|)^{-M}$$ for all $M$.
\end{prop}

\begin{proof}
One can write $$f*g_k(x) = \int_G \left(f(x \cdot y^{-1}) - f(x)\right) g_k(y) dy$$ since $\int_G g(y) dy = 0$ implies $\int_G g_k(y) dy = 0$. Then taking absolute values and using Proposition \ref{prop:fg}, one yields the desired claim.
\end{proof}

Similarly, suppose $f_k(x):= 2^{kQ} f(2^k \cdot x)$. Using the representation $$f_k*g(x) = \int_G f_k(y) g(y^{-1} \cdot x) dy,$$ and invoking Proposition \ref{prop:meanvalueright} instead of Proposition \ref{prop:meanvalue}, we can estimate $f_k*g$ as well.

\begin{prop} \label{prop:fgconv2}
Suppose $$\left|\nabla_b^R g(x) \right| \leq (1+\|x\|)^{-M} \quad \text{and} \quad \left|f(x)\right| \leq (1+\|x\|)^{-M}$$ for all $M > 0$. Suppose further that $\int_G f(y) dy = 0$. Then for any non-negative integer $k$, we have $$|f_k*g(x)| \leq C 2^{-k} (1+\|x\|)^{-M}$$ for all $M$.
\end{prop}

Finally, let $\sigma$ be a non-negative integer, and adopt the shorthand $x_{\sigma} := 2^{-\sigma} \cdot x^{\sigma},$ where $x^{\sigma} := [2^\sigma x_1, x_2, \dots, x_n]$ if $x = [x_1,\dots,x_n]$.
We will need the following mean-value type inequality for $\|x_{\sigma}\|$.

\begin{prop} \label{prop:xtheta}
For any $x, \theta \in G$,
$$
\left|\,\, \|(x \cdot \theta)_{\sigma}\| -  \|x_{\sigma}\| \,\, \right| \leq C \|\theta\|.
$$
Here the constant $C$ is independent of $\sigma$. Similarly $\left|\,\, \|(\theta \cdot x)_{\sigma}\| -  \|x_{\sigma}\| \,\, \right| \leq C \|\theta\|$.
\end{prop}

In particular, taking $\sigma = 0$, the norm function satisfies
\begin{equation} \label{eq:normmeanvalue}
\left|\, \|x \cdot \theta\| - \|x\| \,\right| \leq C \|\theta\|
\end{equation}
for all $x, \theta \in G$.

\begin{proof}[Proof of Proposition~\ref{prop:xtheta}]
We prove the desired inequalities using scale invariance. The key is that the function $x \mapsto \|x_{\sigma}\|$ is homogeneous of degree 1 and smooth away from $0$; in fact, $(\lambda \cdot x)_{\sigma} = \lambda \cdot (x_{\sigma})$, and the homogeneity of the above map follows: $$\|(\lambda \cdot x)_{\sigma}\| = \|\lambda \cdot (x_{\sigma})\| = \lambda \|x_{\sigma}\|.$$ By scaling $x$ and $\theta$ simultanenously, without loss of generality, we may assume that $\|x\|=2$. Now to prove the first inequality, we consider two cases: $\|\theta\| \leq 1$ and $\|\theta\| \geq 1$. If $\|\theta\| \leq 1$, then the (Euclidean) straight line joining $x$ and $x \cdot \theta$ stays in a compact set not containing $0$. Then we apply the Euclidean mean-value inequality to the function $x \mapsto \|x_{\sigma}\|$, which is smooth in this compact set and satisfies $| \nabla \|x_{\sigma}\| | \lesssim 1$ there uniformly in $\sigma$. It follows that if $\|\theta\| \leq 1$, we have $$|\, \|(x \cdot \theta)_{\sigma}\| - \|x_{\sigma}\| \, | \lesssim |\theta| \leq C \|\theta\|.$$ Here $|\theta|$ is the Euclidean norm of $\theta$. On the other hand, if $\|\theta\| \geq 1$, then $$|\, \|(x \cdot \theta)_{\sigma}\| - \|x_{\sigma}\| \, | \leq \|(x \cdot \theta)_{\sigma}\| + \|x_{\sigma}\| \leq \|x \cdot \theta\| + \|x\| \lesssim C (\|x\| + \|\theta\|) \lesssim \|\theta\|,$$ where the second to last inequality follows from the quasi-triangle inequality. Thus we have the desired inequality either case. One can prove the second inequality similarly.
\end{proof}

\section{Left- and right-invariant derivatives} \label{sect:leftright}

Next we describe how one mediates between left- and right-invariant derivatives when working with convolutions on $G$. First we have the following basic identities.

\begin{prop} \label{prop:deriv}
$$X_k (f*g) = f*(X_k g), \quad (X_k f)*g = f*(X_k^R g), \quad \text{and} \quad X_k^R(f*g) = (X_k^R f)*g,$$ assuming $f$, $g$ and their derivatives decay sufficiently rapidly at infinity.
\end{prop}

A proof can be found in Folland-Stein \cite[Page 22]{MR657581}.
Since our groups are not abelian in general, one has to be careful with these identities; one does not have, for instance, the identity between $X_k (f*g)$ and $(X_k f)*g$.

We also have the following flexibility of representing coordinate and left-invariant derivatives in terms of right-invariant ones.

\begin{prop} \label{prop:repR}
\begin{enumerate}[(a)]
\item \label{repRa}
For $1 \leq i \leq n$, the coordinate derivative $\frac{\partial}{\partial x_i}$ can be written as $$\frac{\partial}{\partial x_i} = \sum_{k=1}^{n_1} X_k^R D_{i,k}$$ where $D_{i,k}$ are homogeneous differential operators of degree $j-1$ if $n_{j-1} < i \leq n_j$.
\item \label{repRb}
In fact any $\frac{\partial}{\partial x_i}$, $1 \leq i \leq n$, can be written as a linear combination of $(\nabla_b^R)^{\alpha}$ with coefficients that are polynomials in $x$, where $\alpha$ ranges over a finite subset of the indices $\{1,\dots,n_1\}^{\mathbb{N}}$.
\item \label{lefttoright} If $X$ is a left-invariant vector field of degree 1, then for any Schwartz function $\phi$, there exists $n_1$ Schwartz functions $\hat{\phi}^{(1)}, \dots, \hat{\phi}^{(n_1)}$ such that
$$X\phi = X^R \phi + \sum_{j=1}^{n_1} X_j^R \hat{\phi}^{(j)},$$
with $$\int_G \hat{\phi}^{(j)}(y) dy = 0$$ for all $1 \leq j\leq n_1$. Schematically, we write
$$ X \phi = X^R \phi + \nabla_b^R \hat{\phi} $$
with $\int_G \hat{\phi}(y) dy = 0$.
\end{enumerate}
\end{prop}

\begin{proof}
The crux of the matter here is that our homogeneous groups are stratified. (\ref{repRa}) is rather well-known; see e.g. Proposition (1.26) of Folland-Stein~\cite{MR657581}, and the discussion that follows there; a similar statement with its proof can also be found in Stein \cite[Page 608, Lemma in Section 3.2.2]{MR1232192}.

(\ref{repRb}) follows immediately by iterating (\ref{repRa}).

Finally, to prove (\ref{lefttoright}), note that by (\ref{eq:repleftinvder}) and its analog for right-invariant derivatives, 
$$
X - X^R = \sum_{p=2}^m \sum_{n_{p-1} < k' \leq n_p} Q_{k'}(x) \frac{\partial}{\partial x_{k'}},
$$
where $Q_{k'}(x)$ is a homogeneous polynomial of degree $p-1$ if $n_{p-1} < k' \leq n_p$ (we suppress, in the notation, the dependence of $Q_{k'}$ on $X$ in order to simplify notations). Hence multiplication by $Q_{k'}(x)$ commutes with $\frac{\partial}{\partial x_{k'}}$, and using (\ref{repRa}) for $\frac{\partial}{\partial x_{k'}}$, we get
$$
X - X^R = \sum_{p=2}^m \sum_{n_{p-1} < k' \leq n_p} \sum_{j=1}^{n_1} X_j^R D_{k',j} Q_{k'}(x).
$$ 
As a result, for any Schwartz functions $\phi$, if we write
$$
\hat{\phi}^{(j)} := \sum_{p=2}^m \sum_{n_{p-1} < k' \leq n_p}  D_{k',j} [Q_{k'}(x) \phi(x)]
$$
for $1 \leq j \leq n_1$, then 
$$
X\phi - X^R \phi = \sum_{j=1}^{n_1} X_j^R \hat{\phi}^{(j)};
$$
also, it is easy to check that
$$
\int_G \hat{\phi}^{(j)} = 0 \quad \text{for all $1 \leq j \leq n_1$}
$$
since all $D_{k',j}$ arising in the definition of $\phi^{(j)}$ is a homogeneous differential operator of degree $\geq 1$. (\ref{lefttoright}) then follows.
\end{proof}

Next, we have the following lemma that allows one to write the left-invariant derivative of a bump function as sums of right-invariant derivatives of some other bumps.

\begin{prop} \label{prop:leftright}
Suppose $\phi$ is a Schwartz function on $G$ (by which we mean a Schwartz function on the underlying $\mathbb{R}^n$).
\begin{enumerate}[(a)]
\item \label{prop:leftrighta} If $\int_G \phi(x) dx = 0$, then there exists Schwartz functions $\varphi^{(1)}, \dots, \varphi^{(n_1)}$ such that $$\phi = \sum_{k=1}^{n_1} X_k^R \varphi^{(k)}.$$
\item \label{prop:leftrightb}
If furthermore $\int_G x_k \phi(x) dx = 0$ for all $1 \leq k \leq n_1$, then one can take the $\varphi^{(k)}$'s such that $\int_G \varphi^{(k)}(x) dx = 0$ for all $1 \leq k \leq n_1$. This will be the case, for instance, if $\phi$ is the left-invariant derivative of another Schwartz function whose integral is zero.
\end{enumerate}
\end{prop}

\begin{proof}
To prove part (a), first we claim that any Schwartz function $\phi$ on $\mathbb{R}^n$ that has integral zero can be written as $$\phi = \sum_{i=1}^n \frac{\partial \phi^{(i)}}{\partial x_i}$$ for some Schwartz functions $\phi^{(1)}, \dots, \phi^{(n)}$.

To see that we have such a representation, we use the Euclidean Fourier transform on $\mathbb{R}^n$. First, we observe that since the integral of $\phi$ is zero, which implies $\widehat{\phi}(0) = 0$, we have, for all $\xi \in \mathbb{R}^n$,
\begin{equation} \label{eq:smallxi}
\widehat{\phi}(\xi) = \int_0^1 \frac{d}{ds} \widehat{\phi}(s\xi) ds = \sum_{i=1}^n \xi_i \int_0^1 \frac{\partial \widehat{\phi}}{\partial \xi_i}(s\xi) ds.
\end{equation}
Taking inverse Fourier transform, one can write $\phi$ as a sum of coordinate derivatives of some functions. The problem is that $\int_0^1 \frac{\partial \widehat{\phi}}{\partial \xi_i}(s\xi) ds$, while smooth in $\xi$, does not decay as $\xi \to \infty$. So the above expression is only good for small $\xi$. But for large $\xi$, we have
\begin{equation} \label{eq:largexi}
\widehat{\phi}(\xi) = \sum_{i=1}^n \xi_i \frac{\xi_i}{|\xi|^2} \widehat{\phi}(\xi).
\end{equation} (Here $|\xi|$ is the Euclidean norm of $\xi$.) Hence if we take a smooth cut-off $\eta \in C^{\infty}_c(\mathbb{R}^n)$ with $\eta \equiv 1$ near the origin, then combining (\ref{eq:smallxi}) and (\ref{eq:largexi}), we have
\begin{align*}
\widehat{\phi}(\xi) &= \eta(\xi) \widehat{\phi}(\xi) + (1-\eta(\xi)) \widehat{\phi}(\xi) \\
&= \sum_{i=1}^n \xi_i \left( \eta(\xi) \int_0^1 \frac{\partial \widehat{\phi}}{\partial \xi_i}(s\xi) ds + (1-\eta(\xi))  \frac{\xi_i}{|\xi|^2} \widehat{\phi}(\xi) \right).
\end{align*}
Taking inverse Fourier transform, we get the desired decomposition in our claim above.

Now we return to the group setting. Let $\phi$ be a function on $G$ with integral zero. Using the above claim, identifying $G$ with the underlying $\mathbb{R}^n$, we write $\phi$ as $$\phi = \sum_{i=1}^n \frac{\partial \phi^{(i)}}{\partial x_i}$$ for some Schwartz functions $\phi^{(1)}, \dots, \phi^{(n)}$. Now by Proposition \ref{prop:repR} (a), for all $1 \leq i \leq n$, one can express the coordinate derivatives $\frac{\partial}{\partial x_i}$ in terms of the
%linear combination of
right-invariant derivatives of order 1. Hence by rearranging the above identity we obtain Schwartz functions $\varphi^{(1)}, \dots, \varphi^{(n_1)}$ such that $$\phi = \sum_{k=1}^{n_1} X_k^R \varphi^{(k)},$$ as was claimed in (a).

Finally, if we had in addition $\int_G x_k \phi(x) dx = 0$ for all $1 \leq k \leq n_1$, then one can check, in our construction above, that $\int_G \phi^{(i)}(x) dx = 0$ for all $1 \leq i \leq n_1$. It follows that $\int_G \varphi^{(i)}(x) dx = 0$ for all $1 \leq i \leq n_1$; in fact $\varphi^{(i)}$ is just $\phi^{(i)}$ plus a sum of derivatives of Schwartz functions, which integrates to zero. The rest of the proposition then follows.
\end{proof}

We point out here that in the above two propositions, left-invariant derivatives could have worked as well as right-invariant ones. More precisely:

\begin{prop} \label{prop:repL}
Proposition \ref{prop:repR} and \ref{prop:leftright} remain true if one replaces all right-invariant derivatives by their left-invariant counterparts.
\end{prop}

In what follows, we will develop the habit of consistently denoting the operator $f \mapsto f*K$ by $Kf$ if $K$ is a kernel. If $K$ is a Schwartz function, then $\nabla_b (Kf) = f*(\nabla_b K)$, where (each component of) $\nabla_b K$ is a Schwartz function with integral~0. Thus Proposition \ref{prop:leftright} can be applied to $\nabla_b K$; then one gets some kernels $\tilde{K}^{(k)}$'s that are Schwartz functions, and satisfy $$\nabla_b K = \sum_{k=1}^{n_1} X_k^R \tilde{K}^{(k)}.$$ Schematically we write $\nabla_b K = \nabla_b^R \tilde{K}$, and conclude that $$\nabla_b (Kf) = f*(\nabla_b^R \tilde{K}) = (\nabla_b f) * \tilde{K}.$$ Again writing $\tilde{K}f$ for $f*\tilde{K}$, we obtain the identity $$\nabla_b (Kf) = \tilde{K} (\nabla_b f).$$ If in addition $\int_G K(y) dy = 0$, then one also has $\int_G \tilde{K}(y) dy = 0$, by Proposition \ref{prop:leftright}(\ref{prop:leftrightb}). (The above could also be deduced easily from Proposition~\ref{prop:repR}(\ref{lefttoright}).)

\section{Littlewood-Paley theory and a refinement} \label{sect:LP}

We now turn to the Littlewood-Paley theory for $G$. We need actually two versions of that. First, let $\Psi$ be a Schwartz function on $G$ such that $\int_G \Psi(x) dx = 1$, and such that $\int_G x_k \Psi(x) dx = 0$ for all $1 \leq k \leq n_1$. Such a function exists; in fact one can just take a Schwartz function $\Psi$ on $\mathbb{R}^n$ whose Euclidean Fourier transform is identically 1 near the origin, and think of that as a function on $G$. Now let $\Delta(x) = 2^Q \Psi(2 \cdot x)- \Psi(x)$, and $\Delta_j(x) = 2^{jQ} \Delta(2^j \cdot x)$. Also write $\Delta_jf = f*\Delta_j$. We record here that the assumption $\int_G x_k \Psi(x) dx = 0$ for all $1 \leq k \leq n_1$ guarantees that
\begin{equation} \label{eq:momentDelta}
\int_G x_k \Delta(x) dx = 0
\end{equation} for all $1 \leq k \leq n_1$.

\begin{prop} \label{prop:Lpconv}
If $f \in L^p$ for some $1 < p < \infty$, then $T_K f := \sum_{|j| \leq K} \Delta_j f$ converges to $f$ in $L^p$ norm as $K \to \infty$. In other words,
$$
f = \sum_{j=-\infty}^{\infty} \Delta_j f,
$$
where the convergence is in $L^p$ norm.
\end{prop}

To prove this, we need the following convergence result in $L^p$:

\begin{prop} \label{prop:conv1}
Suppose $\Phi$ is a Schwartz function. Write $\Phi_j(x) = 2^{jQ} \Phi(2^j \cdot x)$. If $f \in L^p$ for some $1 < p < \infty$, then:
\begin{enumerate}[(a)]
\item \label{Sftozero} $f*\Phi_j$ converges to $0$ in $L^p$ norm as $j \to -\infty$.
\item \label{Sftof} If  $\int_G \Phi(y) dy = 1$, then $f*\Phi_j$ converges to $f$ in $L^p$ norm as $j \to +\infty$.
\item \label{Sftozero2} If  $\int_G \Phi(y) dy = 0$, then $f*\Phi_j$ converges to $0$ in $L^p$ norm as $j \to +\infty$.
\end{enumerate}
\end{prop}

Applying Proposition~\ref{prop:conv1} (\ref{Sftozero}) and (\ref{Sftof}) with $\Phi = \Psi$ yields Proposition~\ref{prop:Lpconv}.

\begin{proof}[Proof of Proposition~\ref{prop:conv1}]
It suffices to prove these claims when $f$ is continuous with compact support, since such functions are dense in $L^p$, $1 < p < \infty$, and the maps $f \mapsto f*\Phi_j$ are uniformly bounded on $L^p$ as $j$ varies over the integers. Suppose $f$ is continuous with compact support. Then for $1 < p < \infty$, 
$$\|f*\Phi_j\|_{L^p} \leq \|f\|_{L^1} \|\Phi_j\|_{L^p} = \|f\|_{L^1} 2^{jQ(1-1/p)} \to 0$$ as $j \to -\infty$, proving (\ref{Sftozero}). 

On the other hand, if $\int_G \Phi(y) dy = 1$, then
$$
f*\Phi_j(x) - f(x) = \int_G [f(x \cdot (2^{-j} \cdot y)^{-1}) - f(x)] \Phi(y) dy,
$$
so
$$
\|f*\Phi_j - f \|_{L^p} \leq \int_G \|f(x \cdot (2^{-j} \cdot y)^{-1}) - f(x)\|_{L^p(dx)} |\Phi(y)| dy \to 0
$$
as $j \to +\infty$, by uniform continuity of $f$, and that $f$ has compact support. This proves (\ref{Sftof}).

Finally, if instead $\int_G \Phi(y) dy = 0$, then
$$
f*\Phi_j(x) = \int_G [f(x \cdot (2^{-j} \cdot y)^{-1}) - f(x)] \Phi(y) dy,
$$
so
$$
\|f*\Phi_j\|_{L^p} \leq \int_G \|f(x \cdot (2^{-j} \cdot y)^{-1}) - f(x)\|_{L^p(dx)} |\Phi(y)| dy \to 0
$$
as $j \to +\infty$, by uniform continuity of $f$, and that $f$ has compact support. This proves (\ref{Sftozero2}).
\end{proof}

To put the Proposition~\ref{prop:conv1} in context, note that if $\Phi$ is as in the proposition, but $f$ is merely a (tempered) distribution, it is not necessarily true that $f*\Phi_j \to 0$ in the sense of distributions as $j \to -\infty$. In fact, if $f$ is a non-zero constant, then $f*\Phi_j = f$ for all $j \in \mathbb{Z}$, which does not converge to zero in the sense of distributions. 

We now turn to the analog of such results for functions in $\dot{NL}^{1,p}$:

\begin{prop} \label{prop:NL1pconv}
If $f$ is a function with $\nabla_b f \in L^p$ for some $1 < p < \infty$, then 
$$
\| \nabla_b (f - T_K f) \|_{L^p} \to 0
$$
as $K \to \infty$, where $T_K f$ is as defined in Proposition~\ref{prop:Lpconv}.
\end{prop}

\begin{proof}
First, note that if $X$ is any left-invariant vector field of degree 1, then by Proposition~\ref{prop:repR}(\ref{lefttoright}),
$
X(f*\Psi_j) = f*(X \Psi_j) = f*(X^R \Psi_j + \nabla_b^R \hat{\Psi}_j)
$
where $\int_G \hat{\Psi}(y) dy = 0$. It follows that
$$
X(f*\Psi_j) = (Xf)*\Psi_j + (\nabla_b f)*\hat{\Psi}_j.
$$
Applying Proposition~\ref{prop:conv1}(\ref{Sftof}) for the first term, and Proposition~\ref{prop:conv1}(\ref{Sftozero2}) for the second, we see that 
$$X(f*\Psi_j) \to Xf \quad \text{in $L^p$ norm, as $j \to +\infty$}.$$
Also, by Proposition~\ref{prop:conv1}(\ref{Sftozero}), we have
$$X(f*\Psi_j) \to 0 \quad \text{in $L^p$ norm, as $j \to -\infty$}.$$
Hence our desired conclusion follows.
\end{proof}

Next, we turn to some Littlewood-Paley inequalities. First, for $f \in L^p$, $1 < p < \infty$, we have 
\begin{equation} \label{eq:LPforward}
\left\| \left( \sum_{j=-\infty}^{\infty} |\Delta_j f|^2 \right)^{\frac{1}{2}} \right\|_{L^p} \leq C_p \|f\|_{L^p}.
\end{equation}
This holds because $\int_G \Delta(y) dy = 0$. In fact we have the following more refined Littlewood-Paley theorem:

\begin{prop} \label{prop:refinedLP}
If $D$ is a Schwartz function on $G$, $\int_G D(x) dx = 0$, and $A$ is a constant such that
$$|D(x)| \leq A \left(1+ \|x\|\right)^{-(Q+2)},$$
$$\left| \frac{\partial}{\partial x_k} D(x) \right| \leq A (1 + \|x\|)^{-(Q+r+1)} \quad \text{if $n_{r-1} < k \leq n_r$,} \quad 1 \leq r \leq m,$$
then defining $D_jf = f*D_j$ where $D_j(x) = 2^{jQ} D(2^j \cdot x)$, we have
$$
\left\| \left( \sum_{j=-\infty}^{\infty} |D_j f|^2 \right)^{\frac{1}{2}} \right\|_{L^p} \leq C_p A \|f\|_{L^p}, \quad 1 < p < \infty
$$
where $C_p$ is a constant independent of the kernel $D$.
\end{prop}

Later we will need the fact that the constant on the right hand side of the Littlewood-Paley inequality depends only on $A$ but not otherwise on the kernel $D$. Applying this proposition to $\Delta_j$ yields our claim (\ref{eq:LPforward}).

\begin{proof}
Without loss of generality we may assume $A = 1$. The proof of this proposition relies on a vector-valued singular integral theory on $G$, which is presented, for instance, in \cite[Chapter 13, Section 5.3]{MR1232192}. By our assumptions, it is readily checked that
\begin{equation} \label{eq:D1}
\left(\sum_{j=-\infty}^{\infty} |D_j(x)|^2 \right)^{1/2} \leq C \|x\|^{-Q},
\end{equation}
and
\begin{equation} \label{eq:D2}
\left(\sum_{j=-\infty}^{\infty} \left|\frac{\partial}{\partial x_k} D_j(x)\right|^2 \right)^{1/2} \leq C \|x\|^{-Q-r} \qquad \text{if $n_{r-1} < k \leq n_r$}, \quad 1 \leq r \leq m;
\end{equation}
in fact by scale invariance, it suffices to check this when $\|x\| \simeq 1$. For example, to bound $\sum_{j=-\infty}^{\infty} |\frac{\partial}{\partial x_k} D_j(x)|^2$ where $n_{r-1} < k \leq n_r$, it suffices to split the sum into $\sum_{j \geq 0}$ and $\sum_{j < 0}$; for the second sum, one bounds each term by $C  2^{j(Q+r)}$, and for the first sum, one bounds each term by $C  2^{j(Q+r)} 2^{-j(Q+r+1)}$. Putting these together yields the desired bound (\ref{eq:D2}). (\ref{eq:D1}) can be obtained similarly.

Furthermore, we need to check that for any normalized bump function $\Phi$ supported in the unit ball,
\begin{equation} \label{eq:D3}
\left(\sum_{j=-\infty}^{\infty} \left|\int_G D_j(x) \Phi(R \cdot x) dx \right|^2 \right)^{1/2} \leq C \quad \text{for all $R > 0$}.
\end{equation}
By scale invariance we may assume that $R \simeq 1$. Now when $j < 0$, $$\left| \int_G D_j(x) \Phi(R \cdot x) dx \right| \leq 2^{jQ} \|\Phi(R \cdot x)\|_{L^1} \leq C 2^{jQ},$$ since $|D_j(x)| \leq 2^{jQ}$ and $R \simeq 1$. When $j \geq 0$, $$\left| \int_G D_j(x) \Phi(R \cdot x) dx \right| \leq  \int_G | D_j(x)| |\Phi(R \cdot x) - \Phi(0)| dx \leq C \int_G |D_j(x)| \, R \|x\| dx \leq C 2^{-j},$$ with the first inequality following from $\int_G D(x) dx = 0$, and the second inequality following from Proposition \ref{prop:meanvalue}. Putting these together, we get the desired estimate (\ref{eq:D3}).

From (\ref{eq:D1}), (\ref{eq:D2}), (\ref{eq:D3}), the vector-valued singular integral theory mentioned above applies, and this gives the bounds in our current proposition. Since none of the constants $C$ in (\ref{eq:D1}), (\ref{eq:D2}), (\ref{eq:D3}) depend on the kernel $D$, neither does the bound of our conclusion depend on $D$.
\end{proof}

We will now state and prove the reverse Littlewood-Paley inequality. For that, we need the second version of Littlewood-Paley projections, given by the following proposition:

\begin{prop}\label{prop:equivLP}
There are $2n_1$ functions $\Lambda^{(1)}$, $\dots$, $\Lambda^{(2n_1)}$, each having zero integral on $G$, such that if $\Lambda^{(l)}_j (x) = 2^{jQ} \Lambda^{(l)}(2^j \cdot x)$ and $\Lambda^{(l)}_j f = f*\Lambda^{(l)}_j$, then $$\sum_{l=1}^{2n_1}  \left\| \left( \sum_{j=-\infty}^{\infty} |\Lambda^{(l)}_j f|^2 \right)^{\frac{1}{2}} \right\|_{L^p} \simeq_p \|f\|_{L^p},$$ whenever $f \in L^p$ and $1 < p < \infty$.
\end{prop}

Since the sum in $l$ is usually irrelevant for the estimates, we will abuse notation, and simply write $$\left\| \left( \sum_{j=-\infty}^{\infty} |\Lambda_j f|^2 \right)^{\frac{1}{2}} \right\|_{L^p} \simeq_p \|f\|_{L^p}.$$ Note that we do not claim $f = \sum_{l,j} \Lambda_j^{(l)} f$ here.

We remark that the reversed Littlewood-Paley inequality is false if $f$ is not in $L^p$ to begin with. For instance, if $f$ is a non-zero constant on $G$, then $\Lambda_j f = 0$ for all $j$, whereas $\|f\|_{L^p}$ is infinite for any $1 < p < \infty$. Hence one can only control the $L^p$ norm of $f$ by the square function, when one knows a priori that $f$ is already in $L^p$.

\begin{proof}
The key is to construct $2n_1$ Schwartz functions $\Lambda^{(1)}$, $\dots$, $\Lambda^{(2n_1)}$ and another $2n_1$ Schwartz functions $\Xi^{(1)}$, $\dots$, $\Xi^{(2n_1)}$, each having integral zero, such that if $f \in L^p$, $1 < p < \infty$, then
\begin{equation} \label{eq:deltarep}
f = \sum_{l=1}^{2n_1} \sum_{j=-\infty}^{\infty} f* \Lambda^{(l)}_j * \Xi^{(l)}_j
\end{equation}
where $\Lambda^{(l)}_j (x) = 2^{jQ} \Lambda^{(l)}(2^j \cdot x)$, $ \Xi^{(l)}_j(x) = 2^{jQ} \Xi^{(l)}(2^j \cdot x)$ and the convergence is in $L^p$ norm. Once we have such Schwartz functions, we can write, for $f \in L^p$ and $g \in L^{p'}$,
\begin{align*}
(f,g) &= \sum_{l=1}^{2n_1} \sum_{j=-\infty}^{\infty} \left( \Xi^{(l)}_j \Lambda^{(l)}_j f, g \right) \\
&= \sum_{l=1}^{2n_1} \sum_{j=-\infty}^{\infty} \left( \Lambda^{(l)}_j f, \Xi^{(l)*}_j g \right) \\
&\leq  \sum_{l=1}^{2n_1} \left\| \left( \sum_{j=-\infty}^{\infty} |\Lambda^{(l)}_j f|^2 \right)^{1/2} \right\|_{L^p} \left\| \left( \sum_{j=-\infty}^{\infty} |\Xi^{(l)*}_j g|^2 \right)^{1/2} \right\|_{L^{p'}},
\end{align*}
where $(f,g)$ denotes the inner product on $L^2(G)$ and $\Xi^{(l)*}_j$ is the adjoint of $\Xi^{(l)}_j$ with respect to this inner product, which is also given by the convolution against a Schwartz function of integral zero. Here $p'$ is the dual exponent to $p$. Hence if $1 < p < \infty$, we can estimate the $L^{p'}$ norm above by Proposition \ref{prop:refinedLP}, and get
$$
(f,g) \leq C_p\|g\|_{L^{p'}} \sum_{l=1}^{2n_1} \left\| \left( \sum_{j=-\infty}^{\infty} |\Lambda^{(l)}_j f|^2 \right)^{1/2} \right\|_{L^p}
$$
which is the desired reverse inequality since $g \in L^{p'}$ is arbitrary. The forward inequality follows already from Proposition \ref{prop:refinedLP}.

To construct Schwartz functions $\Lambda^{(1)}$, $\dots$, $\Lambda^{(2n_1)}$ and $\Xi^{(1)}$, $\dots$, $\Xi^{(2n_1)}$ such that they have integral zero and they satisfy (\ref{eq:deltarep}), we proceed as follows. Let $\Psi$ be as in the beginning of this section. Then $\Psi*\Psi$ is a Schwartz function (here $*$ is still the group convolution), and $\int_G \Psi*\Psi(x) dx = 1$. Let $\Psi_j(x) = 2^{jQ} \Psi(2^j \cdot x)$. Then by the argument of Proposition~\ref{prop:Lpconv}, for any $f$ in $L^p$, $1 < p < \infty$, we have
\begin{align}
f &= \sum_{j=-\infty}^{\infty} f*(\Psi_j*\Psi_j - \Psi_{j-1}*\Psi_{j-1}) \notag \\
& = \sum_{j=-\infty}^{\infty} f* \left( \Psi_j*(\Psi_j-\Psi_{j-1}) + (\Psi_j-\Psi_{j-1})*\Psi_{j-1} \right) \label{eq:deltarep1} 
\end{align}
with convergence in $L^p$.
Note that in the smaller brackets, we have the $L^1$ dilation of $\Psi_0-\Psi_{-1}$. Now $\Psi_0-\Psi_{-1}$ has integral zero, and the moments $\int_G x_k (\Psi_0(x)-\Psi_{-1}(x))dx = 0$ for $1 \leq k \leq n_1$. Hence by Proposition \ref{prop:leftright} and Proposition \ref{prop:repL}, we can write $\Psi_0-\Psi_{-1}$ as either $$\Psi_0-\Psi_{-1} = \sum_{k=1}^{n_1} X_k^R \varphi^{(k)} \quad \text{or} \quad \Psi_0-\Psi_{-1} = \sum_{k=1}^{n_1} X_k \psi^{(k)}$$ for some Schwartz functions $\varphi^{(1)}, \dots, \varphi^{(n_1)}$ and $\psi^{(1)}, \dots, \psi^{(n_1)}$, all of which have integral zero. Plugging the first identity back to the first term of (\ref{eq:deltarep1}) and the second identity into the second term of (\ref{eq:deltarep1}), and integrating by parts using Proposition \ref{prop:deriv}, we get
$$
f = \sum_{k=1}^{n_1} \sum_{j=-\infty}^{\infty} f*((X_k \Psi)_j*\varphi^{(k)}_j + \psi^{(k)}_j * (X_k^R \Psi)_{j-1}/2 ).
$$
Renaming the functions, we obtain the desired decomposition of $f$ as in (\ref{eq:deltarep}).
\end{proof}

To proceed further, we consider the maximal function on $G$, defined by $$Mf(x) = \sup_{r > 0} \frac{1}{r^Q} \int_{\|y\| \leq r} |f(x \cdot y^{-1})| dy.$$ We need the following properties of $M$:

\begin{prop} \label{prop:max}
\begin{enumerate}[(a)]
\item $M$ is bounded on $L^p$ for all $1 < p \leq \infty$;
\item $M$ satisfies a vector-valued inequality, namely $$\left\| \left( \sum_{j = -\infty}^{\infty} |Mf_j|^2 \right)^{\frac{1}{2}} \right\|_{L^p} \leq C_p \left\| \left( \sum_{j = -\infty}^{\infty} |f_j|^2 \right)^{\frac{1}{2}} \right\|_{L^p},$$ for $1 < p < \infty$.
\item Moreover, if $|\phi(y)| \leq \varphi\left(\|y\|\right)$ for some decreasing function $\varphi$, and $A = \int_G \varphi(\|y\|) dy$, then $|f*\phi(x)| \leq C A Mf(x)$ where $C$ is a constant depending only on $G$ but not on $\phi$.
\item In particular, $$\left\| \left( \sum_{j = -\infty}^{\infty} |\Lambda_j f_j|^2 \right)^{\frac{1}{2}} \right\|_{L^p} \leq C_p \left\| \left( \sum_{j = -\infty}^{\infty} |f_j|^2 \right)^{\frac{1}{2}} \right\|_{L^p}$$ for all $1 < p < \infty$.
\end{enumerate}
\end{prop}

The proof of these can be found in Stein \cite[Chapter 2]{MR1232192}, once we notice that the group $(G,\|\cdot\|,dx)$ satisfies the real-variable structures set out in Chapter 1 of the same monograph.

We also need a Littlewood-Paley inequality for derivatives:

\begin{prop} \label{prop:derivLP}
Suppose $\nabla_b f \in L^p$ for some $1 < p < \infty$. Then $$\left\| \left(\sum_{j=-\infty}^{\infty} |2^j \Delta_j f|^2\right)^{1/2} \right\|_{L^p} \leq C_p \|\nabla_b f\|_{L^p}.$$
\end{prop}

\begin{proof}
Just notice that by Proposition~\ref{prop:leftright}(\ref{prop:leftrightb}) and (\ref{eq:momentDelta}), there exists Schwartz functions $\varphi^{(k)}$, $k=1,2,\dots, n_1$, such that 
$$
\Delta(x) =  \sum_{k=1}^{n_1} X_k^R \varphi^{(k)}(x),
\qquad \text{with} \qquad 
\int_G \varphi^{(k)}(x) dx = 0.
$$
We write schematically $\Delta = \nabla_b^R \varphi$. Then
$$2^j \Delta_j f = 2^j f * (\nabla_b^R \varphi)_j = (\nabla_b f)*\varphi_j.$$ Since $\int \varphi(x) dx = 0$, it follows that
\begin{align*}
\left\|\left(\sum_{j = -\infty}^{\infty} |2^j \Delta_j f|^2 \right)^{1/2} \right\|_{L^p} 
&=\left\|\left(\sum_{j = -\infty}^{\infty} |(\nabla_b f)*\varphi_j|^2 \right)^{1/2} \right\|_{L^p} \\
&\leq C_p \|\nabla_b f\|_{L^p}
\end{align*}
for $1 < p < \infty$, the last inequality following from Proposition~\ref{prop:refinedLP}.
\end{proof}

The following is a Bernstein-type inequality for our Littlewood-Paley decomposition $\Delta_j$:

\begin{prop} \label{prop:Bern}
If $f \in \dot{NL}^{1,Q}$, then for all $j \in \mathbb{Z}$, we have $$\|\Delta_j f\|_{L^{\infty}} \leq C \|\nabla_b f\|_{L^Q}$$ where $C$ is independent of both $f$ and $j$.
\end{prop}

\begin{proof}
Suppose $\nabla_b f \in L^Q$. Using the notations in the proof of Proposition \ref{prop:derivLP}, $$\Delta_j f = (\nabla_b f)*(2^{-j} \varphi_j).$$ Since $\|2^{-j} \varphi_j\|_{L^{Q/(Q-1)}}$ is a constant $C$ independent of $j$, we see that 
$$
\|\Delta_j f\|_{L^{\infty}} \leq C \|\nabla_b f\|_{L^Q}
$$
as desired.
\end{proof}

Finally, we need the ``heat kernels'' which we define as follows. Let $S$ be a non-negative Schwartz function on $G$, which satisfies $$\int_G S(y) dy = 1 \quad \text{and} \quad S(x) \simeq e^{-\|x\|} \quad \text{for all $x \in G$}.$$ (For instance, $S(x) = c e^{-(1+\|x\|^{2m!})^{\frac{1}{2m!}}}$ will do for a suitable $c$, since $$(1+\|x\|^{2m!})^{\frac{1}{2m!}}-\|x\| \to 0$$ as $\|x\| \to \infty$.) We write $$S_j(x) := 2^{jQ} S(2^j \cdot x),$$ and as usual let $S_j f := f*S_j$.

\section{Algebraic preliminaries} \label{sect:alg}

In this section we describe some algebraic structures we use in the proof of Lemma~\ref{lem:approxsub}. First we have the following algebraic identity:

\begin{prop} \label{prop:alg}
For any sequence $\{a_j\}$, one has $$1 = \sum_{j = 1}^N a_j \prod_{1 \leq j' < j} (1-a_{j'}) + \prod_{j=1}^N (1-a_j) \quad \text{for any $N \in \mathbb{N}$.}$$
\end{prop}

\begin{proof}
This is just saying that
\begin{align*}
1 &= a_1 + (1-a_1)\\
&= a_1 + a_2(1-a_1) + (1-a_1)(1-a_2)\\
&= a_1 + a_2(1-a_1) + a_3(1-a_1)(1-a_2) + (1-a_1)(1-a_2)(1-a_3)\\
&= \dots
\end{align*}
\end{proof}

Now if $\{a_j\}$ is a sequence indexed by $j \in \mathbb{Z}$ instead, with $a_j = 0$ for all $|j| > K$ for some positive integer $K$, then by letting $b_j:= a_{K+1-j}$, and applying the previous proposition to $\{b_j\}$ instead of $\{a_j\}$ with $N = 2K+1$, one has
%$$
%1 = \sum_{j=1}^N a_j \prod_{j < j' \leq N} (1-a_{j'}) + \prod_{j=1}^N (1-a_j).
%$$
$$
1 = \sum_{|j| \leq K} a_j \prod_{j < j' \leq K} (1-a_{j'}) + \prod_{|j| \leq K} (1-a_j).
$$
Hence we have:

\begin{prop} \label{prop:algest}
If $\{a_j\}_{j \in \mathbb{Z}}$ is a sequence of numbers satisfying $0 \leq a_j \leq 1$ for all $j$, with only finitely many non-zero terms, then
$$
0 \leq \sum_{j} a_j \prod_{j' > j} (1-a_{j'}) \leq 1.
$$
\end{prop}

Next, suppose we are given a function $h$ on $G$ such that
\begin{equation} \label{eq:fconstr}
h = \sum_j h_j
\end{equation}
for some functions $h_j$, where all $h_j$ satisfy $\|h_j\|_{L^{\infty}} \leq C$, and only finitely many $h_j$'s are non-zero. We will describe a paradigm in which we approximate $h$ by an $L^{\infty}$ function that we will call $\tilde{h}$. In fact, motivated by the algebraic proposition we have above, we let
\begin{equation} \label{eq:Fconstr}
\tilde{h}=\sum_j h_j \prod_{j' > j} (1 - U_{j'})
\end{equation}
where $U_j$ are some suitable non-negative functions such that
\begin{equation} \label{eq:fjGjboundby1}
C^{-1} |h_j| \leq U_j \leq 1 \quad \text{pointwisely for all $j$,}
\end{equation}
and only finitely many $U_j$'s are non-zero. Then at least $\|\tilde{h}\|_{L^{\infty}} \leq C$ because
$$
|\tilde{h}(x)| \leq \sum_j |h_j(x)| \prod_{j' > j} (1 - U_{j'}(x)) \leq C \sum_j U_j(x) \prod_{j' > j} (1 - U_{j'}(x)) \leq C,
$$
where the last inequality follows from Proposition~\ref{prop:algest}. One would now ask whether this could be any sensible approximation of $h$; in particular, let's try to see whether $\|X_k (h-\tilde{h})\|_{L^Q}$ is small, for $k=1,\dots,n_1$. To understand this, write $h = \sum_j h_j$.
Then
$$
h-\tilde{h} = \sum_j h_j \left( 1 - \prod_{j' > j} (1-U_{j'}) \right).
$$
Using Proposition~\ref{prop:alg} to expand the latter bracket and rearranging the resulting sum, we get
\begin{equation} \label{eq:f-Fconstr}
h-\tilde{h} = \sum_j U_j V_j,
\end{equation}
where $V_j$ is defined by
\begin{equation} \label{eq:Hconstr}
V_j = \sum_{j' < j} h_{j'} \prod_{j' < j'' < j} (1-U_{j''}).
\end{equation}
It follows that
$$
X_k (h-\tilde{h}) = \sum_j (X_k U_j)V_j + \sum_j U_j(X_k V_j).
$$
By Proposition~\ref{prop:algest} and (\ref{eq:fjGjboundby1}), we have
\begin{equation} \label{eq:Hbound}
|V_j| \leq C \quad \text{pointwisely for all $j$}.
\end{equation}
This can be shown using the same argument we have used to bound $\|\tilde{h}\|_{L^{\infty}}$.
Furthermore, we have
\begin{equation} \label{eq:gradHbound}
|X_k V_j| \leq C \sum_{j' < j} (|X_k h_{j'}| + |X_k U_{j'}|).
\end{equation}
In fact this follows from
\begin{equation} \label{eq:derivrule}
X_k V_j = \sum_{j' < j} \left( (X_k h_{j'}) - (X_k U_{j'}) V_{j'} \right) \prod_{j' < j'' < j} (1-U_{j''}).
\end{equation}
(\ref{eq:derivrule}) holds because when one computes $X_k V_j$, either the derivative hits $h_{j'}$, or the derivative hits $U_{j'}$ for some $j' < j$; furthermore, the coefficient of $X_k U_{j'}$ in $X_k V_j$ is $$- V_{j'} \prod_{j' < j'' < j} (1-U_{j''}).$$
From (\ref{eq:Hbound}) and (\ref{eq:gradHbound}), it follows that
\begin{equation} \label{eq:gradf-Fbound}
|X_k (h-\tilde{h})| \leq C\left(\sum_j |X_k U_j| + \sum_j U_j \sum_{j' < j} (|X_k h_{j'}| + |X_k U_{j'}|)\right);
\end{equation}
we will hope to estimate this in $L^Q$ norm on $G$, if we choose $U_j$ suitably.

In the following sections, equations (\ref{eq:fconstr}), (\ref{eq:Fconstr}), (\ref{eq:fjGjboundby1}), (\ref{eq:f-Fconstr}), (\ref{eq:Hconstr}), (\ref{eq:Hbound}), (\ref{eq:gradHbound}), (\ref{eq:derivrule}) and (\ref{eq:gradf-Fbound}) will form a basic paradigm of our construction.

\section{Proof of Lemma~\ref{lem:approxsub}: Outline} \label{sect:outline}

We give an outline of the proof of Lemma \ref{lem:approxsub} in this section, and we defer the detailed proof to the next four sections. The proof will be in 4 steps. First, given $\delta > 0$ and  $f \in \dot{NL}^{1,Q}$, we choose a positive integer $K$, such that
$$
T_K f := \sum_{|j| \leq K} \Delta_j f
$$
satisfies 
\begin{equation} \label{eq:TKf1}
\|\nabla_b (f-T_K f)\|_{L^Q} \leq \frac{\delta}{3} \|\nabla_b f\|_{L^Q}.
\end{equation}
The existence of such $K$ is guaranteed by Proposition~\ref{prop:NL1pconv}. Lemma~\ref{lem:approxsub} then reduces to finding an $F \in \dot{NL}^{1,Q} \cap L^{\infty}$, such that 
$$
\|X_k (T_K f - F)\|_{L^Q} \leq \frac{2\delta}{3} \|\nabla_b f\|_{L^Q}, \quad k = 2, \dots, n_1,
$$
with
$$
\|F\|_{L^{\infty}} + \|\nabla_b F\|_{L^Q} \leq C_{\delta} \|\nabla_b f\|_{L^Q}
$$
for some constant $C_{\delta}$ independent of $f$. The advantage of taking this preliminary step is so that we will effectively be dealing with finite sums and products, when we construct our function $F$ below. On the other hand, since the $K$ we picked here by Proposition~\ref{prop:NL1pconv} depends on both $\delta$ and $f$, in what follows, care must be taken to ensure that all constants we pick are uniform in $K$. It is only by doing so that at the end, our constant $C_{\delta}$ can be chosen independent of $f$.

Next, we will show the existence of some positive integer $N$, such that if we define $f_0$ by
\begin{equation} \label{eq:f0def}
f_0 := T_K f - \sum_{|j| \leq K} S_{j+N} \Delta_j f,
\end{equation}
then 
\begin{equation} \label{eq:f01}
\|\nabla_b f_0\|_{L^Q} \leq \frac{\delta}{3} \|\nabla_b f\|_{L^Q}.
\end{equation}
The constant $N$ will be chosen to depend only on $\delta$, but not on $K$ nor $f$.

To proceed further, let $\sigma$ be a large positive integer to be chosen, which will depend only on $N$ but not on $K$ nor $f$. (In particular, $\sigma$ will depend only on $\delta$ but not $f$.) Suppose from now on $f$ satisfies an additional ``smallness'' condition:
\begin{equation} \label{eq:smallassump}
\|\nabla_b f\|_{L^Q} \leq c_G 2^{-NQ} 2^{-\sigma(Q-1)}.
\end{equation}
Here the constant $c_G$ depends only on the group $G$. We will then define the auxiliary controlling functions $\omega_j$ and $\tomega_j$ as follows. For $x = [x_1,\dots,x_n]$, we recall $x^{\sigma} := [2^\sigma x_1, x_2, \dots, x_n]$ and $x_{\sigma} := 2^{-\sigma} x^{\sigma}$. Let $E$ be a Schwartz function, defined by $$E(x) := e^{-(1+\|x_{\sigma}\|^{2m!})^{\frac{1}{2m!}}}.$$ We write $\Lambda$ for the lattice $\{2^{-N} \cdot s \colon s \in \mathbb{Z}^n\}$ of scale $2^{-N}$ in $G$, and for $j \in \mathbb{Z}$, we define $\omega_j$ by
\begin{equation} \label{eq:omegajdef}
\omega_j(x) := \left( \sum_{r \in \Lambda} \left[ S_{j+N}|\Delta_j f|(2^{-j} \cdot r) E(r^{-1} \cdot (2^j \cdot x))\right ]^Q \right)^{1/Q}
\end{equation}
for all $x \in G$. Here $N$ is the positive integer we chose previously. Note that $\omega_j$ is like a discrete convolution, except that we are using the $l^Q$ norm in $r$ rather than the sum in $r$. We also define $\tomega_j$, for $j \in \mathbb{Z}$, by
\begin{equation} \label{eq:tomegajdef}
\tomega_j := 2^{NQ}E_j S_{j+N} |\Delta_j f|
\end{equation}
where $E_j f := f*E_j$, and $E_j(y) := 2^{jQ} E(2^j y)$. $\omega_j$ and $\tomega_j$ will be used to control the Littlewood-Paley pieces $\Delta_j f$ of $f$; in fact respectively they will control $h_j$ and $g_j$ we introduce below. They will also have better derivatives in the $X_2, \dots, X_{n_1}$ directions than in the $X_1$ direction. %For convenience, we will define $\omega_j$ and $\tomega_j$ to be zero when $|j| > K$.

Now let $f_0$ be as defined in (\ref{eq:f0def}). We decompose $T_K f-f_0 = \sum_{|j| \leq K} S_{j+N} \Delta_j f$ into the sum of two functions $g$ and $h$ as follows. Let $R >> \sigma$ be another positive integer to be chosen, which will again depend only on $N$ but not on $K$ nor $f$. Let $\zeta$ be a smooth function on $[0,\infty)$ such that $\zeta \equiv 1$ on $[0,1/2]$, and $\zeta \equiv 0$ on $[1,\infty)$. For $|j| \leq K$, let
$$
\zeta_j(x) := \zeta \left(\frac{2^j \omega_j(x)}{\sum_{-K \leq k < j, \text{ }  k \equiv j (\text{mod } R)} 2^k \omega_k(x)}\right).
$$
We remark here that $\zeta_j(x)$ is not the $L^1$ dilation of $\zeta$, i.e. $\zeta_{j}(x)\neq 2^{jQ}\zeta(2^j \cdot x)$; it is a smooth approximation of the characteristic function of the set  $\{2^{j} \omega_j <\sum_{-K \leq k<j, k\equiv j(\text{mod }R)}2^k\omega_{k}\}$. We then define two functions $h$ and $g$ such that $$h :=  \sum_{j \in \mathbb{Z}} h_j, \quad g := \sum_{j \in \mathbb{Z}} g_j;$$ here we define, for $|j| \leq K$,
$$
h_j (x) := (1-\zeta_j(x)) S_{j+N}(\Delta_j f)(x),
$$
$$
g_j (x) := \zeta_j(x) S_{j+N}(\Delta_j f)(x)
$$
and we define $h_j := g_j := 0$ if $|j| > K$.
It follows that \begin{equation} \label{eq:TKfdecomp} T_K f = f_0 + g + h. \end{equation}
By $\zeta_j$'s definition, we can think of $h$ as the part where morally ``the high frequencies dominate the low frequencies'', and $g$ as the part where morally the reverse happens. We will now approximate $h$ and $g$ separately by functions in $L^{\infty}$.

First, we will approximate $h$ by some $L^{\infty}$ function $\tilde{h}$ using the paradigm of approximation we discussed in the previous section. Namely, we define
$$
\tilde{h} := \sum_{j \in \mathbb{Z}} h_j \prod_{j' > j} (1-U_{j'})
$$
where we define
$$
U_j := (1-\zeta_j) \omega_j, \quad \text{for $|j| \leq K$},
$$
and $U_j := 0$ if $|j| > K$.
We will prove that
\begin{equation} \label{eq:h1}
\|\tilde{h}\|_{L^{\infty}} \leq C,
\end{equation}
\begin{equation} \label{eq:h2}
\|X_k (h-\tilde{h})\|_{L^Q} \leq C_N 2^{-\sigma/Q} R \|\nabla_b f\|_{L^Q} + C_N 2^{\sigma Q} R \|\nabla_b f\|_{L^Q}^2, \quad k = 2, \dots, n_1,
\end{equation}
and
\begin{equation} \label{eq:h3}
\|X_1 (h-\tilde{h})\|_{L^Q} \leq C_N 2^{\sigma Q} R  \|\nabla_b f\|_{L^Q}.
\end{equation}

Furthermore, using the same paradigm, we approximate $g$ by some $\tilde{g} \in L^{\infty}$, where
$$
\tilde{g} := \sum_{c=0}^{R-1} \sum_{j \equiv c (\text{mod } R) }  g_j \prod_{ j' > j \atop j' \equiv c (\text{mod }R)} (1-G_{j'});
$$
here we define
$$
G_j := \sum_{t > 0, j-t \geq -K \atop t \equiv 0 (\text{mod } R)} 2^{-t} \tomega_{j-t} , \quad \text{for $|j| \leq K$},
$$
and we define $G_j:=0$ for $|j| > K$.
We will prove that
\begin{equation} \label{eq:g1}
\|\tilde{g}\|_{L^{\infty}} \leq C R,
\end{equation}
\begin{equation} \label{eq:g2}
\|\nabla_b (g-\tilde{g})\|_{L^Q} \leq C_N 2^{2\sigma(Q+1)} R^2 2^{-R} \|\nabla_b f\|_{L^Q}.
\end{equation}

Note that one can estimate the full $\nabla_b$ of the error here (rather than only the ``good'' derivatives $X_k$ for $k=2,...,n_1$). We will see in later sections that the ``smallness'' assumption (\ref{eq:smallassump}) on our given $f$ is used, in the proofs of (\ref{eq:h1}), (\ref{eq:h2}), (\ref{eq:h3}), (\ref{eq:g1}) and (\ref{eq:g2}). Also, all constants $C$ and $C_N$ in (\ref{eq:h1}), (\ref{eq:h2}), (\ref{eq:h3}), (\ref{eq:g1}) and (\ref{eq:g2}) will be independent of $K$ and $f$.

We now put everything together. Define $F$ by
$$
F:= \tilde{g} + \tilde{h}.
$$
Then by (\ref{eq:h1}) and (\ref{eq:g1}),
\begin{equation} \label{eq:FLinfty} 
\|F\|_{L^{\infty}} \leq C R.
\end{equation}
Also, by (\ref{eq:TKfdecomp}), 
$$
T_K f - F = f_0 + (g-\tilde{g}) + (h-\tilde{h}).
$$
By (\ref{eq:f01}), (\ref{eq:h2}) and (\ref{eq:g2}), for $k = 2, \dots, n_1$,
\begin{align*}
&\|X_k (T_K f-F)\|_{L^Q} \\
\leq & \|\nabla_b f_0\|_{L^Q} + \|X_k (h-\tilde{h})\|_{L^Q} + \|\nabla_b (g-\tilde{g})\|_{L^Q} \\
\leq & \frac{\delta}{3} \|\nabla_b f\|_{L^Q} + C_N 2^{-\sigma/Q} R \|\nabla_b f\|_{L^Q} + C_N 2^{\sigma Q} R \|\nabla_b f\|_{L^Q}^2 \\ &\quad + C_N 2^{2\sigma(Q+1)} R^2 2^{-R} \|\nabla_b f\|_{L^Q}.
\end{align*}
If one now chooses $R = B \sigma$ where $B$ is a constant $> 2(Q+1)$ (say $B = 2(Q+2)$ will do), and chooses $\sigma$ to be sufficiently big with respect to $N$, then this is bounded by
\begin{equation} \label{eq:Xkgoodboundtemp}
\frac{\delta}{2} \|\nabla_b f\|_{L^Q} + A_{\delta} \|\nabla_b f\|_{L^Q}^2
\end{equation}
where $A_{\delta}$ is a constant depending only on $G$ and $\delta$ (but not on $K$ nor $f$). Hence combining with (\ref{eq:TKf1}), we get
\begin{equation} \label{eq:conclu1}
\|X_k (f-F) \|_{L^Q} \leq \frac{5\delta}{6} \|\nabla_b f\|_{L^Q} + A_{\delta} \|\nabla_b f\|_{L^Q}^2, \quad k = 2,\dots, n_1.
\end{equation}
Similarly, one can show that
\begin{equation} \label{eq:conclu2}
\|X_1 (f-F) \|_{L^Q} \leq A_{\delta} \|\nabla_b f\|_{L^Q},
\end{equation}
and from (\ref{eq:FLinfty}), we get
\begin{equation} \label{eq:conclu3}
\|F\|_{L^{\infty}} \leq A_{\delta}.
\end{equation}

Let's summarize what we have got so far. Given $\delta > 0$, we have picked positive integers $N$, $\sigma$, $R$, and a constant $A_{\delta}$, all depending only on $\delta$, such that whenever $f \in \dot{NL}^{1,Q}$ satisfies (\ref{eq:smallassump}), namely $\|\nabla_b f\|_{L^Q} \leq c_G 2^{-NQ} 2^{-\sigma(Q-1)}$, then (\ref{eq:conclu1}), (\ref{eq:conclu2}), (\ref{eq:conclu3}) holds.

Finally, if now a general $f$ in $\dot{NL}_1^Q$ is given, and $\delta$ is sufficiently small, one will rescale $f$ (by multiplication by a small constant) so that $$\|\nabla_b f\|_{L^Q} = \min\{c_G 2^{-NQ} 2^{-\sigma(Q-1)}, \delta (6A_{\delta})^{-1}\}.$$ This is possible because the right hand side is a number depending only on $\delta$. More precisely, let's call the number on the right hand side above $c_{\delta}$, and let 
$$
f_* = \frac{c_{\delta} f}{\|\nabla_b f\|_{L^Q}},
$$ 
so that 
$$\|\nabla_b f_*\|_{L^Q} = c_{\delta}.$$ In particular, then $f_*$ satisfies our smallness assumption (\ref{eq:smallassump}).
One then construct, using the process described above, an approximation $F_*$ for this rescaled $f_*$. From (\ref{eq:conclu1}), for $k = 2, \dots, n_1$, we have $$\|X_k (f_*-F_*)\|_{L^Q} \leq \delta \|\nabla_b f_*\|_{L^Q}$$ as desired. Also, by (\ref{eq:conclu2}),
$$
\|\nabla_b (f_*-F_*)\|_{L^Q} \leq A_{\delta} \|\nabla_b f_*\|_{L^Q},
$$
and by (\ref{eq:conclu3}),
$$\|F_*\|_{L^{\infty}} \leq A_{\delta} \frac{\|\nabla_b f_*\|_{L^Q}}{\min\{c_G 2^{-NQ} 2^{-\sigma(Q-1)}, \delta (2A_{\delta})^{-1}\}}  = C_{\delta} \|\nabla_b f_*\|_{L^Q}.$$
If one now undo the rescaling we have done, and let 
$$
F :=  \frac{\|\nabla_b f\|_{L^Q} F_*}{c_{\delta}},
$$
then the above shows that $F$ satisfies all the desired conclusions in Lemma~\ref{lem:approxsub}. Hence this completes the proof of Lemma~\ref{lem:approxsub}, modulo the proof of the estimates (\ref{eq:f01}), (\ref{eq:h1}), (\ref{eq:h2}), (\ref{eq:h3}), (\ref{eq:g1}), (\ref{eq:g2}). These will be proved in the next four sections.

In the sequel, $C$ will denote constant independent of $\delta$, $K$, $N$, $\sigma$ and $R$. All constants will be independent of $K$, and all dependence of constants on $N$, $\sigma$ and $R$ will be made clear in the notations.

\section{Estimating $f_0$} \label{sect:f0}

We now begin the proof of our approximation Lemma~\ref{lem:approxsub}. We fix $\delta > 0$, and let $f \in \dot{NL}^{1,Q}$, $K \in \mathbb{N}$. Let $N$ be a large positive integer to be chosen. Define $f_0$ as in (\ref{eq:f0def}). First, 
$$f_0 = \sum_{|j| \leq K} (I - S_{j+N}) \Delta_j f = \sum_{|j| \leq K} \sum_{k \geq N} (S_{j+k+1} - S_{j+k})\Delta_j f$$ where $I$ is the identity operator, and the convergence in the second equality is in $\dot{NL}^{1,Q}$. Now let $P$ be the kernel of the operator $S_1-S_0$. Then $P$ is a Schwartz function, and $$\int_G P(y) dy = 0.$$ Furthermore, if we define $P_k f = f*P_k$ where $P_k(y) = 2^{kQ} P(2^k \cdot y)$, then $$f_0 = \sum_{|j| \leq K} \sum_{k \geq N} P_{j+k}(\Delta_j f),$$ with convergence in $\dot{NL}^{1,Q}$. Using the notation at the end of Section~\ref{sect:leftright}, one gets
$$
\nabla_b f_0 = \sum_{|j| \leq K} \sum_{k \geq N} \tilde{P}_{j+k} \tilde{\Delta}_j (\nabla_b f),
$$
where the convergence in the sum in $k$ is in $L^Q$, the kernels $\tilde{P}$ and $\tilde{\Delta}$ are Schwartz, and satisfy $$\int_G \tilde{P}(y) dy = \int_G \tilde{\Delta}(y) dy = 0$$ since $\int_G P = \int_G \Delta = 0$.
Now $\nabla_b f_0 \in L^Q$, since it is a finite sum (over $j$) of functions in $L^Q$. It follows from Proposition \ref{prop:equivLP} that
\begin{align*}
\|\nabla_b f_0\|_{L^Q} \simeq& \left\| \left(\sum_{r=-\infty}^{\infty} \left| \Lambda_r \sum_{|j| \leq K} \sum_{k \geq N} \tilde{P}_{j+k} \tilde{\Delta}_j (\nabla_b f)\right|^2 \right)^{1/2} \right\|_{L^Q} \notag.
\end{align*}
To proceed further, we replace, in the right hand side of the above formula, the index $j$ by $j+r$, and then pull out the summation in $j$ and $k$. We can also at this point let $K \to +\infty$ on the right hand side of the formula. Then we obtain the following bound for $\|\nabla_b f_0\|_{L^Q}$, namely
\begin{align}
& \sum_{j=-\infty}^{\infty} \sum_{k \geq N} \left\| \left(\sum_{r=-\infty}^{\infty} \left| \Lambda_{r}  \tilde{P}_{j+r+k} \tilde{\Delta}_{j+r} (\nabla_b f)\right|^2 \right)^{1/2} \right\|_{L^Q} \notag,
\end{align}
which is equal to
\begin{align}
& \sum_{j=-\infty}^{\infty} \sum_{k \geq N} \left\| \left(\sum_{r=-\infty}^{\infty} \left| \Lambda_{r+j}  \tilde{P}_{r+k} \tilde{\Delta}_{r} (\nabla_b f)\right|^2 \right)^{1/2} \right\|_{L^Q} \label{eq:f_0est}
\end{align}
if we first replace the index $r$ by $r-j$, and then replace $j$ by $-j$. Now we split the sum into two parts, one where $j < 0$, and one where $j \geq 0$, and show that both of them are bounded by $C 2^{-N} \|\nabla_b f\|_{L^Q}$. The sum where $j < 0$ can be estimated using Proposition \ref{prop:fgconv2}: in fact $\Lambda_{r+j}  \tilde{P}_{r+k} \tilde{\Delta}_{r} (\nabla_b f) =(\tilde{\Delta}_{r} (\nabla_b f)) * (\tilde{P}_{r+k} * \tilde{\Lambda}_{r+j})$, and by Proposition \ref{prop:fgconv2}, $$|\tilde{P}_{r+k} * \tilde{\Lambda}_{r+j}|(x) = \left|( \tilde{P} * \tilde{\Lambda}_{k-j} )_{r+j} \right|(x) \leq C 2^{-(k-j)} 2^{Q(r+j)} \left( 1 + 2^{r+j} \|x\| \right)^{-(Q+1)}$$ since $k-j>0$, $\Lambda$, $\tilde{P}$ are Schwartz, and $\int_G \tilde{P}(y) dy = 0$. From this we infer, using Proposition~\ref{prop:max}~(c), that $$|\Lambda_{r+j} \tilde{P}_{r+k} \tilde{\Delta}_r(\nabla_b f)| \leq C 2^{-(k-j)} M \tilde{\Delta}_r(\nabla_b f).$$ It follows that the sum where $j < 0$ is bounded by
$$
\sum_{j < 0} \sum_{k \geq N} C 2^{-(k-j)} \left\| \left(\sum_{r = -\infty}^{\infty} (M \tilde{\Delta}_r(\nabla_b f))^2 \right)^{1/2} \right\|_{L^Q} \leq C 2^{-N} \|\nabla_b f\|_{L^Q}
$$
as desired. Next, for the sum where $j \geq 0$, one defines an auxiliary kernel $D$ by $$D = \tilde{\Delta} * \tilde{P}_k * \Lambda_j \quad \text{if $j \geq 0$ and $k \geq N$}.$$ We claim that
\begin{prop} \label{prop:D}
$D$ is a Schwartz function,
$$|D(x)| \leq C 2^{-j-k} (1+\|x\|)^{-(Q+2)},$$
$$\left|\frac{\partial}{\partial x_l} D(x)\right| \leq C 2^{-j-k} (1+\|x\|)^{-(Q+r+1)} \quad \text{if $n_{r-1} < l \leq n_r$}, \quad 1 \leq r \leq m,$$
and
$$\int_G D(x) dx = 0.$$
\end{prop}
Assume the proposition for the moment. Then one can apply the refinement of Littlewood-Paley theory in Proposition~\ref{prop:refinedLP} and conclude that
$$
\left\| \left(\sum_{r=-\infty}^{\infty} |D_r (\nabla_b f)|^2 \right)^{1/2} \right\|_{L^Q} \leq C 2^{-j-k} \|\nabla_b f\|_{L^Q},
$$
i.e.
$$
\left\| \left(\sum_{r=-\infty}^{\infty} |\Lambda_{r+j} \tilde{P}_{r+k} \tilde{\Delta}_r (\nabla_b f)|^2 \right)^{1/2} \right\|_{L^Q} \leq C 2^{-j-k} \|\nabla_b f\|_{L^Q}.
$$
Summing over $j \geq 0$ and $k \geq N$, we see that the sum over $j \geq 0$ in (\ref{eq:f_0est}) is bounded by $C 2^{-N} \|\nabla_b f\|_{L^Q}$ as well. This can be made $\leq \frac{\delta}{3} \|\nabla_b f\|_{L^Q}$, by picking $N$ to be a large constant depending only on $\delta$ (but not on $K$ nor $f$). Hence our desired estimate (\ref{eq:f01}) for $f_0$ follows.

\begin{proof}[Proof of Proposition \ref{prop:D}]
It is clear that $D$ is Schwartz and $\int_G D(x) dx = 0$. To prove the estimate for $|D(x)|$, first consider $\tilde{\Delta}*\tilde{P}_k$ with $k \geq N$. From Proposition \ref{prop:deriv}, we have $\nabla_b^R (\tilde{\Delta}*\tilde{P}_k) = (\nabla_b^R \tilde{\Delta})*\tilde{P}_k$, so by Proposition \ref{prop:fgconv}, we get
$$
|(\nabla_b^R)^{\alpha} (\tilde{\Delta}*\tilde{P}_k)(x)| \leq C_{\alpha,M} 2^{-k} (1+\|x\|)^{-M}
$$
for all multi-index $\alpha$ and all $M > 0$. It follows that
$$
|\tilde{\Delta}*\tilde{P}_k(x)| + \left|\frac{\partial}{\partial x_l} (\tilde{\Delta}*\tilde{P}_k)(x)\right| \leq C_{M,l} 2^{-k} (1+\|x\|)^{-M}
$$
for any $M > 0$ and all $1 \leq l \leq n$. Here we applied Proposition \ref{prop:repR} (b), which says that each $\frac{\partial}{\partial x_l}$ can be written as a linear combination of $(\nabla_b^R)^{\alpha}$ with coefficients that are polynomials in $x$. Applying Proposition \ref{prop:fgconv} again, we get
$$
|D(x)| \leq C_{M} 2^{-j-k} (1+\|x\|)^{-M}
$$
for all $M > 0$, since $j \geq 0$. Similarly
$$
|(\nabla_b^R)^{\alpha} D(x)| \leq C_{\alpha,M} 2^{-j-k} (1+\|x\|)^{-M}
$$
for all multi-index $\alpha$ and all $M > 0$. It follows that
$$
\left|\frac{\partial}{\partial x_l} D(x)\right| \leq C_{M} 2^{-j-k} (1+\|x\|)^{-M}
$$
for all $1 \leq l \leq n$ and all $M > 0$, from which our proposition follows.
\end{proof}

\section{Properties of $\omega_j$ and $\tomega_j$}

Suppose $\delta > 0$ is given, and let $N$ be chosen as in the previous section. Let $\sigma$ be a very large positive integer, to be chosen depending only on $N$ (and thus only on $\delta$). Suppose $f \in \dot{NL}^{1,Q}$ is given, and (\ref{eq:smallassump}) holds, i.e. $\|\nabla_b f\|_{L^Q} \leq c_G 2^{-NQ} 2^{-\sigma(Q-1)}$, where $c_G$ is a sufficiently small constant depending only on $G$. %We will also fix an integer $K \in \mathbb{N}$, but all estimates below will be uniform in $K$. 

Define $\omega_j$ and $\tomega_j$ by (\ref{eq:omegajdef}) and (\ref{eq:tomegajdef}) as in Section~\ref{sect:outline}, namely
$$
\omega_j(x) := \left( \sum_{r \in \Lambda} \left[ S_{j+N}|\Delta_j f|(2^{-j} \cdot r) E(r^{-1} \cdot (2^j \cdot x))\right ]^Q \right)^{1/Q}
$$
and
$$
\tomega_j(x) := 2^{NQ}E_j S_{j+N} |\Delta_j f|(x)
$$
%{\color{red} for $|j| \leq K$. In the Propositions below, we estimate below $\omega_j$ and $\tomega_j$ when $|j| \leq K$.}

First, we want a pointwise bound for $\omega_j$. To obtain that we observe:

\begin{lemma} \label{prop:shift}
Let $S_j$ and $E_j$ be defined as in Section~\ref{sect:outline}. Then whenever $x, \theta \in G$ with $\|\theta\| \leq 2^{-j}$, we have $$S_j(x \cdot \theta) \simeq S_j(x) \quad \text{and} \quad E_j(\theta \cdot x) \simeq E_j(x).$$ In particular, we have $$S_jf(x \cdot \theta) \simeq S_jf(x)$$ if $f$ is a non-negative function and $\|\theta\| \leq 2^{-j}$.
\end{lemma}

\begin{proof}
First we observe that $$E(x) \simeq e^{-\|x_{\sigma}\|} \quad \text{for all $x \in G$}.$$ This is because $$(1+\|x_{\sigma}\|^{2m!})^{\frac{1}{2m!}}-\|x_{\sigma}\| \to 0$$ as $\|x_{\sigma}\| \to \infty$. Now by Proposition \ref{prop:xtheta}, we have $$\left| \, \|(\theta \cdot x)_{\sigma}\| - \|x_{\sigma}\| \, \right| \leq C \quad \text{if $\|\theta\| \leq 1$.}$$ Hence from $E(x) \simeq e^{-\|x_{\sigma}\|}$ and $E(\theta \cdot x) \simeq e^{-\|(\theta \cdot x)_{\sigma}\|}$, we get $$E(\theta \cdot x) \simeq E(x) \quad \text{if $\|\theta\| \leq 1$.}$$ Scaling yields the desired claim for $E_j$.

Next, suppose $\|\theta\| \leq 1$. We claim that $S(x \cdot \theta) \simeq S(x)$ for all $x \in G$. This holds because $S(x \cdot \theta) \simeq e^{-\|x \cdot \theta\|}$ and $S(x) \simeq e^{-\|x\|}$ for all $x$, and one can apply (\ref{eq:normmeanvalue}) to compare the latter.
Scaling yields the claim for $S_j$.
\end{proof}

Now comes the pointwise bound for $\omega_j$, from both above and below.

\begin{prop} \label{prop:omegaequiv}
$$\omega_j(x) \simeq \left(\sum_{r \in \Lambda} \left[S_{j+N}|\Delta_j f|(x \cdot 2^{-j}r) E(r^{-1})\right]^Q\right)^{1/Q}. $$
Here the implicit constant is independent of $N$ and $\sigma$.
\end{prop}

\begin{proof}
Recall that by (\ref{eq:omegajdef}),
\begin{align*}
\omega_j(x) &:= \left( \sum_{r \in \Lambda} \left[ S_{j+N}|\Delta_j f|(2^{-j} \cdot r) E(r^{-1} \cdot (2^j \cdot x))\right ]^Q \right)^{1/Q} \\
& = \left( \sum_{s \in (2^j \cdot x)^{-1} \cdot \Lambda} \left[ S_{j+N}|\Delta_j f|(x \cdot (2^{-j} \cdot s)) E(s^{-1})\right ]^Q \right)^{1/Q}
\end{align*}
The last identity follows from a change of variable: if $s = (2^j \cdot x)^{-1} \cdot r$, then we have $r = (2^j \cdot x) \cdot s$, so $2^{-j} \cdot r = x \cdot (2^{-j} \cdot s)$, the last identity following because dilations are group homomorphisms (c.f. (\ref{eq:dilauto})). Now recall that $\Lambda$ is the lattice $\{2^{-N} \cdot s \colon s \in \mathbb{Z}^n\}$. Hence every $s \in (2^j \cdot x)^{-1} \cdot \Lambda$ can be written uniquely as $r \cdot (2^{-N} \cdot \theta)$ for some $r \in \Lambda$ and $\theta \in G$, such that if $\theta = [\theta_1, \dots, \theta_n]$, then all $\theta_k \in [0,1)$. This defines a map from the shifted lattice $(2^j \cdot x)^{-1} \cdot \Lambda$ to the original lattice $\Lambda$, and it is easy to see that this map is a bijection. Hence if the inverse of this map is denoted by $s = s(r)$, then
$$
\omega_j(x) = \left( \sum_{r \in \Lambda} \left[ S_{j+N}|\Delta_j f|(x \cdot (2^{-j} \cdot s(r))) E(s(r)^{-1})\right ]^Q \right)^{1/Q}.
$$
But $s(r) = r \cdot (2^{-N} \cdot \theta)$ for some $\|\theta\| \leq 1$. Thus by Lemma \ref{prop:shift}, we get
$$
E(s(r)^{-1}) \simeq E(r^{-1}).
$$
Also, from the same relation between $s(r)$ and $r$, we have $2^{-j} \cdot s(r) = (2^{-j} \cdot r) \cdot (2^{-(j+N)} \cdot \theta)$ with $\|2^{-(j+N)} \cdot \theta\| \leq 2^{-(j+N)}$. Thus by Lemma \ref{prop:shift} again, we get
$$
S_{j+N}|\Delta_j f|(x \cdot (2^{-j} \cdot s(r))) \simeq S_{j+N} |\Delta_j f|(x \cdot (2^{-j} \cdot r)).
$$
Hence the proposition follows.
\end{proof}

By a similar token, one can prove that
\begin{prop}
$$
\tomega_j(x) \simeq \sum_{r \in \Lambda} S_{j+N}|\Delta_j f|(x \cdot (2^{-j} \cdot r)) E(r^{-1}),
$$
with implicit constants independent of $N$ and $\sigma$.
\end{prop}

\begin{proof}
This is because by (\ref{eq:tomegajdef}),
\begin{align*}
\tomega_j(x) &= 2^{NQ} \int_G S_{j+N}|\Delta_j f|(x \cdot (2^{-j} \cdot y)) E(y^{-1}) dy \\
&= \sum_{r \in \Lambda} \int_{[0,1)^n} S_{j+N}|\Delta_j f|(x \cdot (2^{-j} \cdot r) \cdot(2^{-(j+N)} \cdot \theta)) E((2^{-N} \cdot \theta)^{-1} \cdot r^{-1}) d\theta\\
\end{align*}
The second equality follows from the fact that every $y \in G$ can be written uniquely as $r \cdot (2^{-N} \cdot \theta)$ for some $r \in \Lambda$ and $\theta \in [0,1)^n$, which we have already used in the proof of Proposition~\ref{prop:omegaequiv}. Note again that if $y = r \cdot (2^{-N} \cdot \theta)$, then by the fact that dilations are group homomorphisms, we have $2^{-j} \cdot y = 2^{-j} \cdot (r \cdot (2^{-N} \cdot \theta)) = (2^{-j} \cdot r) \cdot (2^{-(j+N)} \cdot \theta)$. Also, we used $dy = 2^{-NQ} d\theta$ in the change of variables.
Now one can mimic the proof of Proposition~\ref{prop:omegaequiv}. In fact, one observes that whenever $\|\theta\| \leq 1$, one has
$$
E((2^{-N} \cdot \theta)^{-1} \cdot r^{-1}) \simeq E(r^{-1})
$$
and
$$
S_{j+N}|\Delta_j f|(x \cdot (2^{-j} \cdot r) \cdot(2^{-(j+N)} \cdot \theta)) \simeq S_{j+N} |\Delta_j f|(x \cdot (2^{-j} \cdot r)).
$$
One then concludes that
$$\tomega_j(x) \simeq \sum_{r \in \Lambda} S_{j+N}|\Delta_j f|(x \cdot (2^{-j} \cdot r)) E(r^{-1}).$$
This completes the proof.
\end{proof}

From the two propositions above, it follows that

\begin{prop} \label{prop:omegapointwise}
$$S_{j+N}|\Delta_j f|(x) \leq C \omega_j(x) \leq C \tomega_j(x).$$
\end{prop}

\begin{proof}
The first inequality holds because the term corresponding to $r = 0$ in the right hand side of the equation in Proposition~\ref{prop:omegaequiv} is precisely $S_{j+N}|\Delta_j f|(x)$. The second inequality holds by the previous two propositions, since the $l^Q$ norm of a sequence is always smaller than or equal to its $l^1$ norm.
\end{proof}

Next we have

\begin{prop} \label{prop:omegauniform}
$$\|\omega_j\|_{L^{\infty}} \leq C \|\tomega_j\|_{L^{\infty}} \leq C 2^{NQ} 2^{\sigma(Q-1)} \|\nabla_b f\|_{L^Q} \leq 1$$
if $c_G$ in assumption (\ref{eq:smallassump}) is chosen sufficiently small.
\end{prop}

We fix this choice of $c_G$ from now on.

\begin{proof}
The first inequality follows from the previous proposition. The second inequality follows from
$$\|\tomega_j\|_{L^{\infty}} \leq 2^{NQ} \|E\|_{L^1} \|S\|_{L^1} \|\Delta_j f\|_{L^{\infty}},$$
$$\|E\|_{L^1} = C 2^{\sigma(Q-1)},$$
and Bernstein's inequality as in Proposition~\ref{prop:Bern}.
The last inequality holds since $\|\nabla_b f\|_{L^Q} \leq c_G 2^{-NQ} 2^{-\sigma(Q-1)}$ and $c_G$ is sufficiently small.
\end{proof}

\begin{prop} \label{prop:omegaderiv}
$$|X_1 \omega_j| \leq C 2^j \omega_j$$ and
$$|X_k \omega_j| \leq C 2^{j-\sigma} \omega_j \quad \text{for $k = 2, \dots, n_1$}.$$
\end{prop}

\begin{proof}
One just needs to recall the definition of $\omega_j$ from (\ref{eq:omegajdef}), namely
$$
\omega_j(x) = \left( \sum_{r \in \Lambda} \left[ S_{j+N}|\Delta_j f|(2^{-j} \cdot r) E(r^{-1} \cdot (2^j \cdot x))\right ]^Q \right)^{1/Q}
$$
and to differentiate it. Here it is crucial that the variable $x$ is in the argument of $E$ and not in $S_{j+N}|\Delta_jf|$; in other words, we could not have taken the expression in Proposition \ref{prop:omegaequiv} to be the definition of $\omega_j$, because while it is true that the continuous convolution $f\ast g$ can be written as $\int_G f(y^{-1})g(y\cdot x)dy$ or $\int_G f(x\cdot y^{-1})g(y)dy$ via integration by parts, the analogous statement fails for discrete convolutions. Hence if $\omega_j$ was defined by the expression in Proposition~\ref{prop:omegaequiv}, then there would be no way of integrating by parts and letting the derivatives fall on $E$ here.

More precisely, first we observe
$$
|X_k (1+\|x_{\sigma}\|^{2m!})^{\frac{1}{2m!}}|  \leq C 2^{-\sigma} \quad \text{if $k = 2, \dots, n_1$}
$$
and
$$
|X_1 (1+\|x_{\sigma}\|^{2m!})^{\frac{1}{2m!}}|  \leq C.
$$
Thus
$$
|X_1 E(x)| \leq C E(x), \quad \text{and} \quad |X_k E(x)| \leq C 2^{-\sigma} E(x) \quad \text{if $k = 2, \dots, n_1$}.
$$
Now since we are using left-invariant vector fields, they commute with left-translations. It follows that
$$
X_k (E(r^{-1} \cdot (2^j \cdot x))) = 2^j (X_k E)(r^{-1} \cdot (2^j \cdot x))
$$
for all $k = 1, \dots, n_1$, and using the above estimates for $X_k E$, one easily obtains the desired inequalities.
\end{proof}

\begin{prop} \label{prop:Dtomegaj}
$$|X_1 \tomega_j| \leq C 2^j \tomega_j$$ and
$$|X_k \tomega_j| \leq C 2^{j-\sigma} \tomega_j \quad \text{for $k = 2, \dots, n_1$}.$$
\end{prop}

\begin{proof}Note that
$\tomega_j$ can be written as $$\tomega_j=2^{NQ}\int_{G}S_{j+N}|\Delta_jf|(2^{-j} \cdot y^{-1} )E(y \cdot (2^{j} \cdot x))dy.$$
The proof is then almost identical to the previous proposition.
\end{proof}

\begin{prop} \label{prop:maxomega}
$$\|\sup_{j \in \mathbb{Z}} (2^j \omega_j) \|_{L^Q} \leq C_N 2^{\sigma(Q-1)/Q} \|\nabla_b f\|_{L^Q}.$$
\end{prop}

\begin{proof}
This is because
\begin{align*}
\int_{G} (\sup_{j \in \mathbb{Z}} 2^j \omega_j)^Q (x) dx
\leq & \sum_{j \in \mathbb{Z}} \int_{G} (2^j \omega_j)^Q (x) dx  \\
\simeq & \sum_{j \in \mathbb{Z}} \sum_{r \in \Lambda} E(r^{-1})^Q \int_{G} [ 2^j S_{j+N}|\Delta_j f|(x \cdot 2^{-j}r)]^Q dx,
\end{align*}
the last line following from Proposition~\ref{prop:omegaequiv}. Now by the translation invariance of the Lebesgue measure (which is the Haar measure on $G$), the integral in the last sum is independent of $r$. Furthermore,
$$\sum_{r \in \Lambda} E(r)^Q \simeq \sum_{r \in \Lambda} \int_{\theta \in [0,1)^N} E(r \cdot (2^{-N} \cdot \theta))^Q d\theta = 2^{NQ} \int_G E(y)^Q dy \leq C 2^{NQ} 2^{\sigma(Q-1)};$$ here we used Lemma~\ref{prop:shift} in the first inequality, that every $y \in G$ can be written uniquely as $r \cdot (2^{-N} \cdot \theta)$ for some $r \in \Lambda$ and $\theta \in [0,1)^N$ in the middle identity, and that $\|E^Q\|_{L^1} \leq C2^{\sigma(Q-1)}$ in the last inequality. Altogether, this shows
\begin{align*}
\int_{G} (\sup_{j \in \mathbb{Z}} 2^j \omega_j)^Q (x) dx
\leq & C 2^{NQ} 2^{\sigma(Q-1)} \int_{G} \sum_{j \in \mathbb{Z}} \left(2^j |\Delta_j f|(x) \right)^Q dx \\
\leq & C 2^{NQ} 2^{\sigma(Q-1)} \int_{G} \left(\sum_{j \in \mathbb{Z}} \left(2^j |\Delta_j f|(x)\right)^2\right)^{Q/2} dx \\
\leq & C 2^{NQ} 2^{\sigma(Q-1)} \|\nabla_b f\|_{L^Q}^Q,
\end{align*}
the last inequality following from Proposition~\ref{prop:derivLP}. %{\color{red} Note that the constants here are independent of $K$.}
\end{proof}

\section{Estimating $h-\tilde{h}$} \label{sect:h0}

In this section we estimate $h-\tilde{h}$. First, we recall our construction: we have $h = \sum_{j} h_j$, where
$$
h_j (x): =  (1-\zeta_j(x)) S_{j+N}(\Delta_j f)(x) \quad \text{if $|j| \leq K$,}
$$
and $h_j := 0$ if $|j| > K$.
We also have
$$
\tilde{h} = \sum_{j} h_j \prod_{j' > j} (1-U_{j'})
$$
where
$$
U_j := (1-\zeta_j)\omega_j  \quad \text{if $|j| \leq K$,}
$$
and $U_j := 0$ if $|j| > K$.
We will estimate $\tilde{h}$ following our paradigm of approximation in Section~\ref{sect:alg}. By Proposition \ref{prop:omegapointwise} and \ref{prop:omegauniform}, we have $$C^{-1} |h_j| \leq U_j \leq 1.$$ It follows from Proposition~\ref{prop:algest} that $\|\tilde{h}\|_{L^{\infty}} \leq C$, proving (\ref{eq:h1}).

Next, following the derivation of (\ref{eq:gradf-Fbound}), we have
$$
|X_k(h-\tilde{h})| \leq C \sum_{|j| \leq K} |X_k U_j| + \sum_{|j| \leq K} |U_j| \sum_{j' < j} (\|\nabla_b h_{j'}\|_{L^{\infty}} + \|\nabla_b U_{j'}\|_{L^{\infty}})
$$
for $k = 1, \dots, n_1$.
But $U_j$ can be estimated by
$$|U_j(x)| \leq \omega_j(x) \chi_{\{ 2^j \omega_j > (1/2) \sum_{ -K \leq k < j,\, k \equiv j (\textrm{mod} R)} 2^k \omega_k \}}(x)$$
where $\chi$ denotes the characteristic function of a set. This is because $1-\zeta_j(x) = 0$ unless $2^j \omega_j(x)> (1/2) \sum_{-K \leq k < j,\, k \equiv j (\textrm{mod} R)} 2^k \omega_k(x)$.
Next, for $k = 2, \dots, n_1$, $X_k U_j$ can be estimated by
$$
|X_k U_j(x)| \leq C 2^{j-\sigma} \omega_j(x) \chi_{\{ 2^j \omega_j > (1/2) \sum_{ -K \leq k < j,\, k \equiv j (\textrm{mod} R)} 2^k \omega_k \}}(x)
$$
because $|X_k \omega_j| \leq C2^{j-\sigma} \omega_j$ by Proposition \ref{prop:omegaderiv}, and 
$$
|X_k \zeta_j| \leq C 2^{j-\sigma}  \chi_{\{ 2^j \omega_j > (1/2) \sum_{ -K \leq k < j,\, k \equiv j (\textrm{mod} R)} 2^k \omega_k \}}.
$$ 
(The last inequality follows by differentiating the definition of $\zeta_j$, and using $|X_k \omega_j| \leq C 2^{j-\sigma} \omega_j$ again.) Similarly,
$$
|X_1 U_j(x)| \leq C 2^{j} \omega_j(x) \chi_{\{ 2^j \omega_j > (1/2) \sum_{ -K \leq  k < j,\, k \equiv j (\textrm{mod} R)} 2^k \omega_k \}}(x).
$$
Finally, we have 
$$
\|\nabla_b h_j\|_{L^{\infty}} + \|\nabla_b U_j\|_{L^{\infty}} \leq C_N 2^{\sigma(Q-1)} 2^j \|\nabla_b f\|_{L^Q}.
$$ 
(The estimate on $\nabla_b U_j$ follows from the above discussion and Proposition~\ref{prop:omegauniform}, while the estimate on $\nabla_b h_j$ is similar.)
So altogether, for $k = 2, \dots, n_1$, we have
\begin{align*}
|X_k(h-\tilde{h})(x)| \leq & C 2^{-\sigma} \mathfrak{S}(x) +  C_N 2^{\sigma(Q-1)} \mathfrak{S}(x) \|\nabla_b f\|_{L^Q},
\end{align*}
where
$$\mathfrak{S}(x):=\sum_{ |j| \leq K } 2^j \omega_j(x) \chi_{\{ 2^j \omega_j > (1/2) \sum_{ -K \leq k < j,\, k \equiv j (\textrm{mod} R)} 2^k \omega_k \}}(x).$$
Similarly,
$$
|X_1(h-\tilde{h})(x)| \leq C_N 2^{\sigma(Q-1)} \mathfrak{S}(x).
$$

To proceed further, we estimate the $L^Q$ norm of the sum $\mathfrak{S}(x)$; this sum can be rewritten as
$$\sum_{c = 0}^{R-1} \sum_{|j| \leq K , j \equiv c (\textrm{mod} R)} 2^j \omega_j(x) \chi_{\{ 2^j \omega_j > (1/2) \sum_{ -K \leq  k < j,\, k \equiv c (\textrm{mod} R)} 2^k \omega_k \}}(x).$$
For each fixed $c$, we have
\begin{align}\label{ineqn:max}
\left|  \sum_{ |j| \leq K , j \equiv c (\textrm{mod} R)} 2^j \omega_j \chi_{\{ 2^j \omega_j > (1/2) \sum_{ -K \leq k < j,\, k \equiv c (\textrm{mod} R)} 2^k \omega_k \}} \right |
\leq 3 \sup_{|j| \leq K } (2^j \omega_j).
\end{align}
This is true because for a fixed $x$, one can pick the biggest integer $j_0$, with $|j_0| \leq K ,\, j_0 \equiv c (\textrm{mod} R)$ such that $2^{j_0} \omega_{j_0} > (1/2) \sum_{ -K \leq  k < j_0,\, k \equiv c (\textrm{mod} R)} 2^k \omega_k$. Then
\begin{align*}
&\sum_{ |j|\leq K,\, j \equiv c (\textrm{mod} R)}2^{j}\omega_{j}\chi_{\{ 2^j \omega_j > (1/2) \sum_{-K \leq k < j,\, k \equiv c (\textrm{mod} R)} 2^k \omega_k \}} \\
=&\sum_{ -K \leq j\leq j_0,\, j \equiv c (\textrm{mod} R)}2^{j}\omega_{j}\chi_{\{ 2^j \omega_j > (1/2) \sum_{-K \leq k < j,\, k \equiv c (\textrm{mod} R)} 2^k \omega_k \}}\\
\leq &2^{j_0}\omega_{j_0}+ \sum_{-K \leq j< j_0,\, j \equiv c (\textrm{mod} R)}2^{j}\omega_{j}\\
\leq &3\cdot 2^{j_0}\omega_{j_0}\\
\leq &3 \sup_{ |j| \leq K } (2^j \omega_j),
\end{align*}
yielding the inequality (\ref{ineqn:max}). Hence
$$\mathfrak{S}(x) \leq 3R \sup_{j} (2^j \omega_j)(x),$$
and from Proposition~\ref{prop:maxomega} we conclude that
$$\|\mathfrak{S}\|_{L^Q} \leq CR 2^{\sigma(Q-1)/Q} \|\nabla_b f\|_{L^Q}.$$

Putting these altogether, for $k = 2, \dots, n_1$, we have
\begin{equation}\label{ineqn:Dh}
\begin{split}
\|X_k(h-\tilde{h})\|_{L^Q}
\leq & C_N 2^{-\sigma/Q} R \|\nabla_b f\|_{L^Q} + C_N 2^{\sigma Q} R \|\nabla_b f\|_{L^Q}^2,\end{split}
\end{equation}
and this proves (\ref{eq:h2}).

Using the pointwise bound of $X_1(h-\tilde{h})$, and applying the same method as in (\ref{ineqn:Dh}), one can prove
$$
\|X_1 (h-\tilde{h})\|_{L^Q} \leq C_N 2^{\sigma Q} R \|\nabla_b f\|_{L^Q},
$$
completing our proof of (\ref{eq:h3}).

\section{Estimating $g-\tilde{g}$} \label{sect:g0}

In this section we estimate $g-\tilde{g}$. Again we recall our construction: we have $g = \sum_{j} g_j$, where
$$
g_j (x) = \zeta_j(x) S_{j+N}(\Delta_j f)(x) \quad \text{if $|j| \leq K$},
$$
and $g_j:=0$ if $|j| > K$. We also have
$$
\tilde{g} = \sum_{c=0}^{R-1} \sum_{j \equiv c (\text{mod } R)} g_j \prod_{j' > j \atop j' \equiv c (\text{mod }R)} (1-G_{j'})
$$
where
$$
G_j = \sum_{t > 0, j-t \geq -K  \atop t \equiv 0 (\text{mod } R)} 2^{-t} \tomega_{j-t} \quad \text{for $|j| \leq K$},
$$
and $G_j:=0$ for $|j| > K$.
Now for $|j| \leq K$, by Proposition \ref{prop:omegapointwise}, 
$$
C^{-1} |g_j| \leq \omega_j \zeta_j.
$$ 
But then 
$$
C^{-1} \omega_j \zeta_j \leq G_j \leq 1.
$$ 
In fact, on the support of $\zeta_j$, 
$$
\omega_j \leq \sum_{-K \leq k < j \atop k \equiv j (\text{mod } R)} 2^{k-j} \omega_k = \sum_{t > 0, j-t \geq -K \atop t \equiv 0 (\text{mod } R)} 2^{-t} \omega_{j-t} \leq C G_j,
$$ 
and the first inequality follows. The last inequality comes from Proposition~\ref{prop:omegauniform}. Thus 
$$
C^{-1} |g_j| \leq G_j \leq 1,
$$ 
and from Proposition~\ref{prop:algest}, we have $|\tilde{g}| \leq CR$. This proves (\ref{eq:g1}).

Next
\begin{align*}
g-\tilde{g}
&= \sum_{c=0}^{R-1} \sum_{j \equiv c (\text{mod } R) \atop |j| \leq K} g_j \left(1 - \prod_{K \geq j' > j \atop j' \equiv c (\text{mod } R)} (1-G_{j'}) \right) \\
&= \sum_{c=0}^{R-1} \sum_{j \equiv c (\text{mod } R) \atop  |j| \leq K} G_j H_j \\
&= \sum_{|j| \leq K} G_j H_j
\end{align*}
where for $|j| \leq K$,
\begin{equation} \label{eq:Hjdef}
H_j := \sum_{-K \leq j' < j \atop j' \equiv j (\text{mod } R)} g_{j'} \prod_{j' < j'' < j \atop j'' \equiv j (\text{mod } R)} (1 - G_{j''}).
\end{equation}
Note that both $G_j$ and $H_j$ are $C^{\infty}$ functions, since the sums and products defining them are only finite.
By Proposition~\ref{prop:algest}, an immediate estimate of $H_j$ is
$$|H_j|\leq \sum_{-K \leq j'<j}|g_{j'}|\prod_{j'<j''<j}(1-G_{j''})\leq C\sum_{-K \leq j'<j}G_{j'}\prod_{j'<j''<j}(1-G_{j''})\leq C.$$

We now collect below some estimates for $\nabla_b g_j$, $\nabla_b G_j$ and $\nabla_b H_j$ for $|j| \leq K$. To begin with, we have
\begin{prop} \label{prop:tomegamaximal}
$$\tomega_j \leq C_N 2^{\sigma(Q+1)} MM(\Delta_j f)$$
where $M$ is the maximal function defined before Proposition \ref{prop:max}.
\end{prop}

\begin{proof}
$$E(x) \leq C (1+\|x_{\sigma}\|)^{-(Q+1)} \leq C 2^{\sigma(Q+1)} (1+ \|x\|)^{-(Q+1)}$$ and the latter is an integrable radially decreasing function. Thus $$\tomega_j = 2^{NQ} E_j(S_{j+N}|\Delta_j f|) \leq C 2^{NQ} 2^{\sigma(Q+1)} MM(\Delta_j f).$$
\end{proof}

\begin{prop} \label{prop:DGj}
$$|\nabla_b G_j| \leq C_N 2^{\sigma(Q+1)} \sum_{t > 0 \atop t \equiv 0 (\textrm{mod } R)} 2^{-t} 2^{j-t} MM(\Delta_{j-t}f).$$
\end{prop}

\begin{proof}
One differentiates the definition of $G_j$ and estimates the derivatives of $\tomega_j$ using Proposition~\ref{prop:tomegamaximal} and~\ref{prop:Dtomegaj}.
\end{proof}

\begin{prop} \label{prop:Dgj}
$$|\nabla_b g_j| \leq C_N 2^j M(\Delta_j f).$$
\end{prop}

\begin{proof}
One differentiates $g_j(x) = \zeta_j(x) S_{j+N}(\Delta_jf)(x)$, letting the derivative hit either $\zeta_j$ or $S_{j+N}$, and estimates the rest by the maximal function. The worst term is when the derivative hits $S_{j+N}$, which gives a factor of $2^{j+N}$.
\end{proof}

\begin{prop} \label{prop:DHj}
$$
|\nabla_b H_j| \leq C_N 2^{\sigma(Q+1)} \sum_{t > 0 \atop t \equiv 0 (\textrm{mod } R)} 2^{j-t} MM(\Delta_{j-t} f).
$$
\end{prop}

\begin{proof}
Following the derivation of (\ref{eq:derivrule}) from (\ref{eq:Hconstr}) in Section~\ref{sect:alg}, and using the definition of $H_j$ in (\ref{eq:Hjdef}), we have
\begin{equation} \label{eq:derivruleH}
\nabla_b H_j = \sum_{ -K \leq j' < j \atop j' \equiv j (\textrm{mod } R)} \left(\nabla_b g_{j'}- (\nabla_b G_{j'}) H_{j'} \right) \prod_{j' < j'' < j \atop j' \equiv j (\textrm{mod } R)} (1 - G_{j''}),
\end{equation}
so
\begin{align*}
|\nabla_b H_j| \leq C \sum_{j' < j \atop j' \equiv j (\textrm{mod } R)} (|\nabla_b g_{j'}| + |\nabla_b G_{j'}|).
\end{align*}
This, with Proposition \ref{prop:DGj} and \ref{prop:Dgj}, leads to
$$C_N 2^{\sigma(Q+1)} \sum_{l > 0 \atop l \equiv 0 (\textrm{mod } R)} \sum_{t \geq 0 \atop t \equiv 0 (\textrm{mod } R)} 2^{-t} 2^{j-l-t} MM(\Delta_{j-l-t}f).$$ Rearranging gives the desired bound.
\end{proof}

The proofs of the next two estimates are the same as those in Proposition \ref{prop:Dgj} and \ref{prop:DGj}, except that one differentiates once more.

\begin{prop} \label{prop:D2gj}
$$|\nabla_b^2 g_j| \leq C_N 2^{2j} M(\Delta_j f).$$
\end{prop}

\begin{prop} \label{prop:D2Gj}
$$|\nabla_b^2 G_j| \leq C_N 2^{\sigma(Q+1)} \sum_{t > 0 \atop t \equiv 0 (\textrm{mod } R)} 2^{-t} 2^{2(j-t)} MM(\Delta_{j-t} f).$$
\end{prop}

Finally we estimate second derivatives of $H_j$:

\begin{prop} \label{prop:D2Hj}
$$|\nabla_b^2 H_j| \leq C_N 2^{2\sigma(Q+1)} \sum_{t > 0 \atop t \equiv 0 (\textrm{mod } R)} 2^{2(j-t)} MM(\Delta_{j-t} f).$$
\end{prop}

\begin{proof}
Differentiating (\ref{eq:derivruleH}) once more, again using the way we derived (\ref{eq:derivrule}) from (\ref{eq:Hconstr}), we get
\begin{align*}
 \nabla_b^2 H_j
= \sum_{-K \leq j' < j \atop j' \equiv j (\textrm{mod } R)} \left( \nabla_b [\nabla_b g_{j'}- (\nabla_b G_{j'}) H_{j'}] - (\nabla_b G_{j'})(\nabla_b H_{j'}) \right) \prod_{j' < j'' < j \atop j' \equiv j (\textrm{mod } R)} (1 - G_{j''}).
\end{align*}
Thus $|G_j|\leq 1, |H_j| \leq C$ imply that
\begin{align*}
|\nabla_b^2 H_j|
\leq  & C \sum_{t > 0, j-t > -K \atop t \equiv 0 (\textrm{mod } R)} |\nabla_b^2 g_{j-t}| + |\nabla_b^2 G_{j-t}| + |\nabla_b G_{j-t}| |\nabla_b H_{j-t}|.
\end{align*}
The first two terms can be estimated using Proposition \ref{prop:D2gj} and \ref{prop:D2Gj}. For the last term, Proposition \ref{prop:DGj} and \ref{prop:DHj} give
$$
|\nabla_b G_j| |\nabla_b H_j| \leq C_N 2^{2\sigma(Q+1)} \sum_{t > 0 \atop t \equiv 0 (\textrm{mod } R)} \sum_{l > 0 \atop l \equiv 0 (\textrm{mod } R)} 2^{-t} 2^{2j-t-l} (MM\Delta_{j-t} f)(MM \Delta_{j-l}f).
$$
Now we split the sum into two parts: one where $t > l$, and the other where $l \geq t$, and use $\|\Delta_j f\|_{L^{\infty}} \leq C \|\nabla_b f\|_{L^Q} \leq 1$. In the first sum, we estimate $MM\Delta_{j-t}f$ by a constant; this is possible because $MM \Delta_{j-t} f$ is bounded by $\|\Delta_{j-t} f\|_{L^{\infty}}$, which is bounded by a constant by Bernstein inequality (Proposition~\ref{prop:Bern}) and our assumption (\ref{eq:smallassump}). We then sum $t$ to get a bound
$$ C\sum_{l > 0 \atop l \equiv 0 (\textrm{mod } R)} 2^{-l} 2^{2(j-l)} MM \Delta_{j-l}f. $$
In the second sum, we estimate $MM\Delta_{j-l}f$ by a constant instead, and sum $l$ to get a bound
$$ C\sum_{t > 0 \atop t \equiv 0 (\textrm{mod } R)} 2^{-t} 2^{2(j-t)} MM \Delta_{j-t}f. $$
These two bounds are identical. So
$$
\sum_{t > 0 \atop t \equiv 0 (\textrm{mod } R)} |\nabla_b G_{j-t}| |\nabla_b H_{j-t}|
\leq C_N 2^{2\sigma(Q+1)} \sum_{t > 0 \atop t \equiv 0 (\textrm{mod } R)} \sum_{l > 0 \atop l \equiv 0 (\textrm{mod } R)} 2^{-l} 2^{2(j-t-l)} MM \Delta_{j-t-l}f.
$$
Rearranging we get the desired bound.
\end{proof}

Now we will estimate $$\|\nabla_b (g-\tilde{g}) \|_{L^Q} = \left\| \sum_{|s| \leq K } \nabla_b (G_s H_s) \right\|_{L^Q}.$$ The argument below will show that $\nabla_b (G_s H_s) \in L^Q$ for all $|s| \leq K$, so we could use the reversed Littlewood-Paley inequality in Proposition~\ref{prop:equivLP}, and bounded this by
\begin{align}
\leq &\left\| \left( \sum_{j = -\infty}^{\infty} \left( \sum_{ |s| \leq K } | \Lambda_j \nabla_b (G_s H_s) | \right)^2 \right)^{1/2} \right\|_{L^Q} \notag \\
=&
\left\| \left( \sum_{j = -\infty}^{\infty} \left( \sum_{s \colon |j-s| \leq K } | \Lambda_j \nabla_b (G_{j-s} H_{j-s}) | \right)^2 \right)^{1/2} \right\|_{L^Q} \notag \\
\leq &
\sum_{s = -\infty}^{\infty} \left\| \left( \sum_{ j \colon |j-s| \leq K } | \Lambda_j \nabla_b (G_{j-s} H_{j-s}) |^2 \right)^{1/2} \right\|_{L^Q} \notag \\
=& \sum_{s = -\infty}^{\infty} \left\| \left( \sum_{|j| \leq K } | \Lambda_{j+s} \nabla_b (G_j H_j) |^2 \right)^{1/2} \right\|_{L^Q}. \label{eq:Gjtobeest}
\end{align}
We split the sum into two parts: $\sum_{s \leq R}$ and $\sum_{s > R}$. We shall pick up a convergence factor $2^{-|s|}$ or $|s|2^{-|s|}$ for each term so that we can sum in $s$.

To estimate the first sum, we fix $s \leq R$. Then for each $|j| \leq K$, we split $G_j$ into a sum
$$G_j = G_j^{(1)} + G_j^{(2)},$$
where $$G_j^{(1)} = \sum_{0 < t < |s|, j-t \geq -K \atop t \equiv 0 (\textrm{mod } R)} 2^{-t} \tomega_{j-t},$$
and $$G_j^{(2)} = \sum_{t \geq \max\{|s|,R\}, j-t \geq -K \atop t \equiv 0 (\textrm{mod } R)} 2^{-t} \tomega_{j-t}.$$
Note that the splitting of $G_j$ depends on $s$; in particular, if $-R \leq s \leq R$, then $G_j^{(1)} = 0$ and $G_j^{(2)} = G_j$.

Now we estimate
$$
\left\| \left( \sum_{ |j| \leq K } | \Lambda_{j+s} \nabla_b (G_j^{(1)} H_j) |^2 \right)^{1/2} \right\|_{L^Q}.
$$
We have
\begin{align*}
\Lambda_{j+s} (\nabla_b (G_j^{(1)} H_j))
&= (\nabla_b (G_j^{(1)} H_j))*\Lambda_{j+s} \\
&= 2^{j+s} (G_j^{(1)}H_j) * (\nabla_b^R \Lambda)_{j+s}
\end{align*}
by the compatibility of convolution with the left- and right-invariant derivatives. Hence from $|H_j| \leq C$, and
$$|G_j^{(1)}| \leq C_N 2^{\sigma(Q+1)} \sum_{0 < t < |s| \atop t \equiv 0 (\textrm{mod } R)} 2^{-t} MM(\Delta_{j-t}f)$$
which follows from Proposition \ref{prop:tomegamaximal}, we have
$$
|\Lambda_{j+s} (\nabla_b (G_j^{(1)} H_j))| \leq C_N 2^{\sigma(Q+1)} 2^s \sum_{0 < t < |s| \atop t \equiv 0 (\textrm{mod } R)} 2^{j-t} MMM(\Delta_{j-t}f).
$$
Taking square function in $j$ and then the $L^Q$ norm in space, we obtain that, when $s < -R$,
\begin{align}
\left\| \left( \sum_{|j| \leq K } | \Lambda_{j+s} \nabla_b (G_j^{(1)} H_j) |^2 \right)^{1/2} \right\|_{L^Q}
& \leq C_N 2^{\sigma(Q+1)} \frac{|s|}{R}  2^s \left\| \left(\sum_{j=-\infty}^{\infty} (2^j |\Delta_j f|)^2 \right)^{1/2} \right\|_{L^Q} \notag \\
&\leq C_N 2^{\sigma(Q+1)} \frac{|s|}{R} 2^s \|\nabla_b f\|_{L^Q} \label{eq:Gj1est}.
\end{align}
Here the last inequality follows from Proposition \ref{prop:derivLP}.
The same norm on the left hand side above is of course zero when $-R \leq s \leq R$.

Next, we estimate
$$
\left\| \left( \sum_{|j| \leq K } | \Lambda_{j+s} \nabla_b (G_j^{(2)} H_j) |^2 \right)^{1/2} \right\|_{L^Q} \leq C \left\| \left( \sum_{ |j| \leq K } |\nabla_b (G_j^{(2)} H_j) |^2 \right)^{1/2} \right\|_{L^Q}.
$$
Now by $|H_j|\leq C$,
$$
|\nabla_b(G_j^{(2)} H_j)| \leq C \left(|\nabla_b G_j^{(2)}| + |G_j^{(2)}| |\nabla_b H_j|\right).
$$
We know
\begin{equation} \label{eq:Gj2bd}
|G_j^{(2)}| \leq C_N 2^{\sigma(Q+1)} \sum_{t \geq \max\{|s|,R\} \atop t \equiv 0 (\textrm{mod } R)} 2^{-t} MM(\Delta_{j-t}f)
\end{equation}
by Proposition~\ref{prop:tomegamaximal}, and $|\nabla_b G_j^{(2)}|$ is bounded by $C \sum_{t \geq \max\{|s|,R\} \atop t \equiv 0 (\textrm{mod } R)} 2^{-t} 2^{j-t} \tomega_{j-t}$ by Proposition~\ref{prop:DGj}. Therefore by Proposition~\ref{prop:tomegamaximal} again, we have
$$
|\nabla_b G_j^{(2)}| \leq C_N 2^{\sigma(Q+1)} \sum_{t \geq \max\{|s|,R\} \atop t \equiv 0 (\textrm{mod } R)} 2^{-t} 2^{j-t} MM(\Delta_{j-t}f).
$$
Hence
\begin{align*}
\left\| \left( \sum_{|j| \leq K } |\nabla_b G_j^{(2)} |^2 \right)^{1/2} \right\|_{L^Q}
\leq & C_N 2^{\sigma(Q+1)} \sum_{t \geq \max\{|s|,R\} \atop t \equiv 0 (\textrm{mod } R)} 2^{-t} \left\|\left( \sum_{j = -\infty}^{\infty} |2^j \Delta_j f |^2 \right)^{1/2} \right\|_{L^Q} \\
\leq & C_N 2^{\sigma(Q+1)} 2^{-\max\{|s|,R\}} \|\nabla_b f\|_{L^Q}.
\end{align*}
Furthermore, by (\ref{eq:Gj2bd}) and Proposition \ref{prop:DHj}, one can estimate
$$
|G_j^{(2)}| |\nabla_b H_j| \leq C_N 2^{2\sigma(Q+1)} \sum_{t \geq \max\{|s|,R\} \atop t \equiv 0 (\textrm{mod } R)} \sum_{m > 0 \atop m \equiv 0 (\textrm{mod } R)} 2^{j-t-m} MM(\Delta_{j-t}f) MM(\Delta_{j-m}f).
$$
We split this sum into the sum over three regions of $t$ and $m$: the first one being where $t \geq \max\{|s|,R\}$ and $m > t$;
the second one being where $t \geq \max\{|s|,R\}$ and $t \geq m\geq \max\{|s|,R\}$, which is equivalent to say  
$m \geq \max\{|s|,R\}$ and $t \geq m$;
and the last one being where $0 < |m| < \max \{|s| , R\}$ and $t \geq \max\{|s|,R\}$. The first two sums are basically the same; each can be bounded by
$$\sum_{m \geq \max\{|s|, R\} \atop m \equiv 0 (\textrm{mod } R) } 2^{j-m} MM(\Delta_{j-m}f) \sum_{t \geq m} 2^{-t} MM(\Delta_{j-t}f),$$
which is bounded by
$$ \sum_{m \geq \max\{|s|, R\} \atop m \equiv 0 (\textrm{mod } R) } 2^{-m} 2^{j-m} MM(\Delta_{j-m}f)$$
since we can bound $MM\Delta_{j-t} f$ by a constant (c.f proof of Proposition~\ref{prop:D2Hj}) and take sum in $t$. The last sum is bounded by
$$C 2^{-\max\{|s|,R\}} \sum_{0 < m < \max\{|s|,R\} \atop m \equiv 0 (\textrm{mod } R) } 2^{j-m} MM (\Delta_{j-m}f)$$
for the same reason.
Thus
\begin{align*}
|G_j^{(2)}| |\nabla_b H_j|
&\leq C_N 2^{2\sigma(Q+1)} \left(2 \sum_{m\geq \max\{|s|, R\}  \atop m \equiv 0 (\textrm{mod } R) }2^{-m} 2^{j-m} MM(\Delta_{j-m}f) \right.\\
&\qquad \qquad + \left. 2^{-\max\{|s|,R\}} \sum_{0 < m < \max\{|s|,R\} \atop m \equiv 0 (\textrm{mod } R)} 2^{j-m} MM (\Delta_{j-m}f) \right).\\
\end{align*}
Taking $l^2$ norm in $j$ and then $L^Q$ norm in space, we get
\begin{align*}
\left\| \left( \sum_{ |j| \leq K } |G_j^{(2)} (\nabla_b H_j) |^2 \right)^{1/2} \right\|_{L^Q}
\leq & C_N 2^{2\sigma(Q+1)} \max\{|s|,R\} 2^{-\max\{|s|,R\}} \|\nabla_b f\|_{L^Q},
\end{align*}
It follows that
\begin{align}
\left\| \left( \sum_{ |j| \leq K } | \Lambda_{j+s} \nabla_b (G_j^{(2)} H_j) |^2 \right)^{1/2} \right\|_{L^Q}
\leq & C_N 2^{2\sigma(Q+1)} \max\{|s|,R\} 2^{-\max\{|s|,R\}} \|\nabla_b f\|_{L^Q}. \label{eq:Gj2est}
\end{align}
Summing (\ref{eq:Gj1est}) and (\ref{eq:Gj2est}) over $s \leq R$, we get a bound
$$
C_N 2^{2\sigma(Q+1)} R^2 2^{-R} \|\nabla_b f\|_{L^Q}
$$
for the first half of the sum in (\ref{eq:Gjtobeest}).

Next we look at the second half of the sum in (\ref{eq:Gjtobeest}), that corresponds to the sum over all $s > R$. First,
\begin{equation} \label{eq:Gjslarge}
|\Lambda_{j+s} \nabla_b (G_j H_j)| \leq |\Lambda_{j+s} ((\nabla_b G_j) H_j)| + |\Lambda_{j+s} (G_j (\nabla_b H_j))|.
\end{equation}
The first term can be written as
\begin{align*}
&\int_G ((\nabla_b G_j)(x \cdot y^{-1}) - (\nabla_b G_j)(x)) H_j(x \cdot y^{-1}) \Lambda_{j+s}(y) dy + (\nabla_b G_j)(x) (\Lambda_{j+s} H_j)(x) \\
=& \, I + II.
\end{align*}
The second term in (\ref{eq:Gjslarge}) can be written as
\begin{align*}
&\int_G (G_j(x \cdot y^{-1}) - G_j(x)) (\nabla_b H_j) (x \cdot y^{-1}) \Lambda_{j+s}(y) dy + G_j(x) \Lambda_{j+s} (\nabla_b H_j)(x) \\
=& \, III + IV.
\end{align*}
We estimate $I$, $II$, $III$, $IV$ separately.

First, in $I$, we bound $|H_j| \leq C$, and write
\begin{align*}
&(\nabla_b G_j)(x \cdot y^{-1}) - (\nabla_b G_j)(x) \\
=& 2^{NQ} \sum_{t > 0 \atop t \equiv 0 (\textrm{mod } R)} 2^{-t} 2^{j-t}  \int_G S_{j+N-t}|\Delta_{j-t} f|(x \cdot z^{-1}) \left((\nabla_b E)_{j-t}(z \cdot y^{-1})-(\nabla_b E)_{j-t}(z)\right) dz.
\end{align*}
We put this back in $I$, and thus need to bound
\begin{equation} \label{eq:nablaEintIII}
\int_G \left|(\nabla_b E)_{j-t}(z \cdot y^{-1})-(\nabla_b E)_{j-t}(z)\right| |\Lambda_{j+s}(y)| dy.
\end{equation}
But
$$|\nabla_b^2 E(x)| \leq C E(x) \leq \frac{C}{(1+2^{-\sigma}\|x\|)^K} \leq C 2^{\sigma K} \frac{1}{(1+\|x\|)^K}$$ for all positive integers $K$.
We will use this estimate with $K = 2(Q+1)$, and apply the remark after Proposition \ref{prop:fg}; the integral (\ref{eq:nablaEintIII}) is then bounded by
\begin{align*}
& C 2^{2\sigma(Q+1)} 2^{-s-t} 2^{(j-t)Q} (1+2^{j-t}\|z\|)^{-(Q+1)}.
\end{align*}
Hence
$$
|I| \leq C_N 2^{2\sigma(Q+1)} 2^{-s} \sum_{t > 0 \atop t \equiv 0 (\textrm{mod } R)} 2^{-2t} 2^{j-t}  MM(\Delta_{j-t} f)(x).
$$
Taking square function in $j$ and $L^Q$ norm in space, we get a bound
$$
C_N 2^{2\sigma(Q+1)} 2^{-s} \|\nabla_b f\|_{L^Q}.
$$

For $II$, recall the pointwise bound for $\nabla_b G_j$ from Proposition~\ref{prop:DGj}:
$$
|\nabla_b G_j| \leq C_N 2^{\sigma(Q+1)} \sum_{t > 0 \atop t \equiv 0(\textrm{mod } R)} 2^{-t} 2^{j-t} MM(\Delta_{j-t}f).
$$
To estimate $\Lambda_{j+s} H_j$, we use part (\ref{prop:leftrighta}) of Proposition \ref{prop:leftright}, and write (schematically) $\Lambda$ as $\nabla_b^R \cdot \Phi$ where $\Phi$ is a ($2n$ tuple of) Schwartz function, and integrate by parts. Then
$$
|\Lambda_{j+s} H_j| = 2^{-j-s} |(\nabla_b H_j)*\Phi_{j+s}| \leq 2^{-j-s} \|\nabla_b H_j\|_{L^{\infty}} \leq C_N 2^{\sigma(Q+1)} 2^{-s},
$$
since $\|\nabla_b H_j\|_{L^{\infty}} \leq C_N 2^{\sigma(Q+1)} 2^j$. Hence
$$
|II| \leq C_N 2^{2\sigma(Q+1)} 2^{-s} \sum_{t > 0 \atop t \equiv 0 (\textrm{mod } R)} 2^{-t} 2^{j-t} MM(\Delta_{j-t}f).
$$
Taking square function in $j$ and $L^Q$ norm in space, we get a contribution
$$
C_N 2^{2\sigma(Q+1)} 2^{-s} \|\nabla_b f\|_{L^Q}.
$$

Now to bound $III$, we follow our strategy as in $I$. First we bound $|\nabla_b H_j| \leq C_N 2^{\sigma(Q+1)} 2^j$ by Proposition~\ref{prop:DHj}, and write
\begin{align*}
& G_j(x \cdot y^{-1}) - G_j(x) \\
=& 2^{NQ} \sum_{t > 0 \atop t \equiv 0 (\textrm{mod } R)} 2^{-t} \int_G S_{j+N-t} |\Delta_{j-t} f|(x \cdot z^{-1}) (E_{j-t}(z \cdot y^{-1}) - E_{j-t}(z)) dz.
\end{align*}
We put this back in $III$, and thus need to bound
\begin{equation} \label{eq:nablaEint}
\int_G \left|E_{j-t}(z \cdot y^{-1})-E_{j-t}(z)\right| |\Lambda_{j+s}(y)| dy.
\end{equation}
But
$$|\nabla_b E(x)| \leq C E(x) \leq \frac{C}{(1+2^{-\sigma}\|x\|)^K} \leq C 2^{\sigma K} \frac{1}{(1+\|x\|)^K}$$ for all postive integers $K$. We will take $K = 2(Q+1)$, and apply the remark after Proposition \ref{prop:fg}; the integral (\ref{eq:nablaEint}) is then bounded by
\begin{align*}
& C 2^{\sigma(Q+1)} 2^{-s-t} 2^{(j-t)Q} (1+2^{j-t}\|z\|)^{-(Q+1)}.
\end{align*}
Hence
$$
|III| \leq C_N 2^{2\sigma (Q+1)} 2^{-s} \sum_{t > 0 \atop t \equiv 0 (\textrm{mod } R)} 2^{-t} 2^{j-t}  MM(\Delta_{j-t} f)(x).
$$
Taking square function in $j$ and $L^Q$ norm in space, we get a bound
$$
C_N 2^{2\sigma (Q+1)}  2^{-s} \|\nabla_b f\|_{L^Q}.
$$

Finally, to estimate $IV$, we recall that $|G_j| \leq 1$, as was shown at the beginning of this section. Furthermore,
\begin{align*}
|(\Lambda_{j+s} (\nabla_b H_j))(x)| \leq & |(\nabla_b H_j) * (\nabla_b^R \Phi)_{j+s}(x)| \\
= & 2^{-j-s} | (\nabla_b^2 H_j) * \Phi_{j+s} (x)| \\
 \leq & 2^{-j-s} M(\nabla_b^2 H_j)(x). \\
\end{align*}
By Proposition \ref{prop:D2Hj}, this is bounded by
\begin{align*}
 &  C_N 2^{2\sigma (Q+1)} 2^{-s} \sum_{t > 0 \atop t \equiv 0 (\textrm{mod } R)} 2^{-t} 2^{j-t} MMM(\Delta_{j-t} f)(x).
\end{align*}
Hence
$$
|IV| \leq C_N 2^{2\sigma (Q+1)}  2^{-s} \sum_{t > 0 \atop t \equiv 0 (\textrm{mod } R)} 2^{-t} 2^{j-t} MMM(\Delta_{j-t} f)(x).
$$
Taking square function in $j$ and then $L^Q$ norm in space, this is bounded by
$$
C_N 2^{2\sigma (Q+1)}  2^{-s} \|\nabla_b f\|_{L^Q}.
$$

Hence
$$
\sum_{s > R} \left\| \left( \sum_{|j| \leq K } | \Lambda_{j+s} \nabla_b (G_j H_j) |^2 \right)^{1/2} \right\|_{L^Q}
\leq C_N 2^{2\sigma (Q+1)}  2^{-R} \|\nabla_b f\|_{L^Q}.
$$
Altogether, (\ref{eq:Gjtobeest}) is bounded by
\begin{align*}
& C_N 2^{2\sigma (Q+1)}  R^2 2^{-R} \|\nabla_b f\|_{L^Q}.
\end{align*}
This proves our claim (\ref{eq:g2}), and marks the end of the proof of our approximation Lemma~\ref{lem:approxsub}.

\section{Proof of Theorem \ref{thm:soldbarbstrong} and \ref{thm:subGNstrong}} \label{sect:thms}

In this section we prove Theorem \ref{thm:soldbarbstrong} and \ref{thm:subGNstrong}. We first recall the $\dbarb$ complex on the Heisenberg group $\mathbb{H}^n$.

First, $\mathbb{H}^n$ is a simply connected Lie group diffeomorphic to $\mathbb{R}^{2n+1}$. We write $[x,y,t]$ for a point on $\mathbb{R}^{2n+1}$, where $x, y \in \mathbb{R}^n$ and $t \in \mathbb{R}$. The group law on the Heisenberg group is then given by $$[x,y,t] \cdot [u,v,w] = [x+u, y+v, t+w+2(yu - xv)],$$ where $yu$ is the dot product of $y$ and $u$ in $\mathbb{R}^n$. The left-invariant vector fields of order 1 on $\mathbb{H}^n$ are then linear combinations of the vector fields $X_1, \dots, X_{2n}$, where
$$X_k = \frac{\partial}{\partial x_k} + 2 y_k \frac{\partial}{\partial t} \quad \text{and} \quad X_{k+n} = \frac{\partial}{\partial y_k} - 2 x_k \frac{\partial}{\partial t} \quad \text{for $k=1,\dots,n$}.$$ Thus in this case, $n_1$ is equal to $2n$, and $$\nabla_b f = (X_1 f, \dots, X_{2n} f).$$ The one-parameter family of automorphic dilations on $\mathbb{H}^n$ is given by $$\lambda \cdot [x,y,t] = [\lambda x, \lambda y, \lambda^2 t] \quad \text{for all $\lambda > 0$}.$$ The homogeneous dimension in this case is $Q = 2n+2$.

Now let $$Z_k = \frac{1}{2} (X_k - i X_{k+n}) \quad \text{and} \quad \Zbar_k = \frac{1}{2} (X_k + i X_{k+n}), \quad k = 1, \dots, n.$$ For $0 \leq q \leq n$, the $(0,q)$ forms on the Heisenberg group $\mathbb{H}^n$ are expressions of the form $$\sum_{|\alpha| = q} u_{\alpha} d\zbar^{\alpha},$$ where the sum is over all strictly increasing multi-indices $\alpha = (\alpha_1, \dots, \alpha_q)$ of length $q$ with letters in $\{1, \dots, n\}$; in other words, each $\alpha_k \in \{1, \dots, n\}$, and $\alpha_1 < \alpha_2 < \dots < \alpha_q$. $d\zbar^{\alpha}$ here is a shorthand for $d\zbar_{\alpha_1} \wedge \dots \wedge d\zbar_{\alpha_q}$, and each $u_{\alpha}$ is a smooth function on $\mathbb{H}^n$. The $\dbarb$ complex is then defined by
$$
\dbarb u = \sum_{k=1}^{2n} \sum_{|\alpha| = q} \Zbar_k (u_{\alpha}) d\zbar_k \wedge d\zbar^{\alpha} , \quad \text{if $\sum_{|\alpha| = q} u_{\alpha} d\zbar^{\alpha}$}.
$$
By making the above $d\zbar^{\alpha}$ an orthonormal basis for $(0,q)$ forms at every point, one then has a Hermitian inner product on $(0,q)$ forms at every point on $\mathbb{H}^n$, with which one can define an inner product on the space of $(0,q)$ forms on $\mathbb{H}^n$ that has $L^2$ coefficients. One can then consider the adjoint of $\dbarb$ with respect to this inner product, namely
$$
\dbarb^*u = \sum_{|\alpha| = q} \sum_{k \in \alpha} -Z_k(u_{\alpha}) d\zbar_k \mathrel{\lrcorner} d\zbar^{\alpha};
$$
here the interior product $\mathrel{\lrcorner}$ is just the usual one on $\mathbb{R}^{2n+1}$.

\begin{proof}[Proof of Theorem~\ref{thm:soldbarbstrong}]
The key idea is that when one computes $\dbarb^*$ of a $(0,q+1)$ form on $\mathbb{H}^n$, only $2(q+1)$ of the $2n$ real left-invariant derivatives of order 1 are involved. So if $q + 1 < n$, then for each component of the $q$ form, there will be some real left-invariant derivatives of degree 1 that are irrelevant in computing $\dbarb^*$, and we can give up estimates in those directions when we apply Lemma \ref{lem:approxsub}.

We will use the bounded inverse theorem and an argument closely related to the usual proof of the open mapping theorem.

Let $\dot{NL}^{1,Q} (\Lambda^{(0,q+1)})$ be the space of $(0,q+1)$ forms on $\mathbb{H}^n$ with $\dot{NL}^{1,Q}$ coefficients, and similarly define $L^Q (\Lambda^{(0,q)})$. Consider the map $\dbarb^* \colon \dot{NL}^{1,Q} (\Lambda^{(0,q+1)}) \to L^Q (\Lambda^{(0,q)})$.
It is bounded and has closed range. Hence it induces a bounded linear bijection between the Banach spaces
$\dot{NL}^{1,Q} (\Lambda^{(0,q+1)}) / \text{ker($\dbarb^*$)}$ and $\text{Image($\dbarb^*$)}
\subseteq L^Q (\Lambda^{(0,q)})$.
By the bounded inverse theorem, this map has a bounded inverse; hence for any $(0,q)$ form $f \in \text{Image($\dbarb^*$)} \subseteq L^Q (\Lambda^{(0,q)})$, there exists $\alpha^{(0)} \in \dot{NL}^{1,Q} (\Lambda^{(0,q+1)}) $ such that
$$
\begin{cases}
\dbarb^*\alpha^{(0)} = f\\
\|\nabla_b \alpha^{(0)}\|_{L^Q} \leq C \|f\|_{L^Q}.
\end{cases}
$$
Now for $q < n-1$, if $I$ is a multi-index of length $q+1$, then one can pick $i \notin I$ and approximate $\alpha^{(0)}_I$
by Lemma \ref{lem:approxsub} in all but the $X_i$ direction; more precisely, for any $\delta > 0$,
there exists $\beta^{(0)}_I \in \dot{NL}^{1,Q} \cap L^{\infty}$ such that
$$
\sum_{j \ne i} \left\| X_j \left(\alpha^{(0)}_I - \beta^{(0)}_I\right) \right\|_{L^Q}
\leq \delta \left\|\nabla_b \alpha^{(0)}_I\right\|_{L^Q} \leq C \delta \left\|f \right\|_{L^Q}
$$
and
$$
\left\| \beta^{(0)}_I \right\|_{L^{\infty}} + \left\| \nabla_b \beta^{(0)}_I \right\|_{L^Q}
\leq A_{\delta} \left\| \nabla_b \alpha^{(0)}_I \right\|_{L^Q} \leq C A_{\delta} \left\|f\right\|_{L^Q}.
$$
Then if $\delta$ is picked so that $C \delta \leq \frac{1}{2}$, we have
$\beta^{(0)} := \sum_I \beta^{(0)}_I d\zbar^I \in \dot{NL}^{1,Q} \cap L^{\infty} (\Lambda^{0,q+1})$ satisfying
$$
\begin{cases}
\|f - \dbarb^* \beta^{(0)}\|_{L^Q} \leq \frac{1}{2} \|f \|_{L^Q} \\
\|\beta^{(0)}\|_{L^{\infty}} + \|\nabla_b \beta^{(0)}\|_{L^n} \leq A \|f\|_{L^Q}
\end{cases}
$$
(the first equation holds because $\|f-\dbarb^*\beta^{(0)}\|_{L^Q} = \|\dbarb^*(\alpha^{(0)} - \beta^{(0)})\|_{L^Q}$, and $A$ here is a fixed constant).
In other words, we have sacrificed the property $f = \dbarb^*\alpha^{(0)}$ by replacing $\alpha^{(0)} \in \dot{NL}^{1,Q}$
with $\beta^{(0)}$, which in addition to being in $\dot{NL}^{1,Q}$ is in $L^{\infty}$. Now we repeat the process,
with $f-\dbarb^*\beta^{(0)}$ in place of $f$, so that we obtain $\beta^{(1)}\in \dot{NL}^{1,Q} \cap L^{\infty} (\Lambda^{0,q+1})$ with
$$
\begin{cases}
\|f-\dbarb^*\beta^{(0)}-\dbarb^*\beta^{(1)}\|_{L^Q} \leq \frac{1}{2} \|f-\dbarb^*\beta^{(0)}\|_{L^Q} \leq \frac{1}{2^2} \|f\|_{L^Q} \\
\|\beta^{(1)}\|_{L^{\infty}} + \|\nabla_b \beta^{(1)}\|_{L^Q} \leq A \|f-\dbarb^*\beta^{(0)}\|_{L^Q} \leq \frac{A}{2} \|f\|_{L^Q}.
\end{cases}
$$
Iterating, we get $\beta^{(k)}\in \dot{NL}^{1,Q} \cap L^{\infty} (\Lambda^{0,q+1})$ such that
$$
\begin{cases}
\|\beta-\dbarb^*(\beta^{(0)}+\dots+\beta^{(k)})\|_{L^Q} \leq \frac{1}{2^{k+1}} \|f\|_{L^Q} \\
\|\beta^{(k)}\|_{L^{\infty}} + \|\nabla_b \beta^{(k)}\|_{L^Q} \leq \frac{A}{2^k} \|f\|_{L^Q}.
\end{cases}
$$
Hence
$$
Y = \sum_{k=0}^{\infty} \beta^{(k)}
$$
satisfies $Y \in \dot{NL}^{1,Q} \cap L^{\infty} (\Lambda^{0,q+1})$ with
$$
\begin{cases}
\dbarb^*Y = f \\
\|Y\|_{L^{\infty}} + \|\nabla_b Y\|_{L^Q} \leq 2A \|f\|_{L^Q}
\end{cases}
$$
as desired.
\end{proof}

We mention that by the duality between $(0,q)$ forms and $(0,n-q)$ forms, we have the following Corollary for solving $\dbarb$ on $\mathbb{H}^n$:

\begin{cor} \label{cor:solvedbarb}
Suppose $q \ne 1$. Then for any $(0,q)$-form $f$ on $\mathbb{H}^n$ that has coefficients in $L^Q$ and that is the $\dbarb$ of some other form on $\mathbb{H}^n$ with coefficients in $\dot{NL}^{1,Q}$, there exists a $(0,q-1)$-form $Y$ on $\mathbb{H}^n$ with coefficients in $L^{\infty} \cap \dot{NL}^{1,Q}$ such that $$\dbarb Y = f$$ in the sense of distributions, with $\|Y\|_{L^{\infty}} + \|\nabla_b Y\|_{L^Q} \leq C \|f\|_{L^Q}.$
\end{cor}

\begin{proof}[Proof of Theorem~\ref{thm:subGNstrong}]
We use duality and the Hodge decomposition for $\dbarb$. Suppose first $u$ is a $C^{\infty}_c$ $(0,q)$ form on $\mathbb{H}^n$ with $2 \leq q \leq n-2$. We test it against a $(0,q)$ form $\phi \in C^{\infty}_c$. Now $$\phi = \dbarb^* \alpha + \dbarb \beta$$ by Hodge decomposition for $\dbarb$ on $\mathbb{H}^n$, where $$\|\nabla_b \alpha\|_{L^Q} + \|\nabla_b \beta\|_{L^Q} \leq C\|\phi\|_{L^Q}.$$ Apply Theorem~\ref{thm:soldbarbstrong} to $\dbarb^* \alpha$ and Corollary~\ref{cor:solvedbarb} to $\dbarb \beta$, we get $$\phi = \dbarb^* \tilde{\alpha} + \dbarb \tilde{\beta}$$ where $\tilde{\alpha}$ and $\tilde{\beta}$ have coefficients in $\dot{NL}^{1,Q} \cap L^{\infty}$, with bounds
$$
\|\nabla_b \tilde{\alpha}\|_{L^Q} + \|\tilde{\alpha}\|_{L^{\infty}} \leq C \|\dbarb^* \alpha\|_{L^Q} \leq C \|\phi\|_{L^Q},
$$
$$
\|\nabla_b \tilde{\beta}\|_{L^Q} + \|\tilde{\beta}\|_{L^{\infty}} \leq C \|\dbarb \beta\|_{L^Q} \leq C \|\phi\|_{L^Q}.
$$
Thus
\begin{align*}
(u,\phi) &= (u, \dbarb^* \tilde{\alpha}) + (u, \dbarb \tilde{\beta}) \\
&= (\dbarb u, \tilde{\alpha}) + (\dbarb^*u, \tilde{\beta}) \\
&\leq \|\dbarb u\|_{L^1 + (\dot{NL}^{1,Q})^*} \|\tilde{\alpha}\|_{L^{\infty} \cap \dot{NL}^{1,Q}} + \|\dbarb^* u\|_{L^1 + (\dot{NL}^{1,Q})^*} \|\tilde{\beta}\|_{L^{\infty} \cap \dot{NL}^{1,Q}} \\
&\leq C (\|\dbarb u\|_{L^1 + (\dot{NL}^{1,Q})^*} + \|\dbarb^* u\|_{L^1 + (\dot{NL}^{1,Q})^*}) \|\phi\|_{L^Q}.
\end{align*}
This proves the desired inequality (\ref{eq:GNdbarbq}).

The proof of (\ref{eq:GNdbarb0}) for functions $u$ orthogonal to the kernel of $\dbarb$ is similar, which we omit.
\end{proof}

\bigskip
\footnotesize
\noindent\textit{Acknowledgments.}
The authors would like to thank E. Stein for originally suggesting this problem, H. Brezis for his encouragement and interest in our work, and the referee who read the manuscript very carefully and gave us many helpful suggestions.

The first author is supported in part by NSF grant DMS 1205350. The second author is supported in part by NSF grant DMS 1201474.


\begin{thebibliography}{BB1}

% Use the widest label as parameter.

%% Reference items may be numbered or have labels of your choice.
%% The author's surname PRECEDES the initial of the first name
%% The issue number is only given when the issues are paginated separately.
%% In book titles, first letters are capitalized.
%% Only journal volume numbers are boldfaced.

%%%%%%%%%%% To ease editing, use normal size:

\normalsize
\baselineskip=17pt

%%%%%%%%%%%%%

\bibitem[AN]{MR2771258}
Amrouche, C., Nguyen, H. H.:
New estimates for the div-curl-grad operators and elliptic problems with {$L^1$}-data in the whole space and in the half-space.
J. Differential Equations. \textbf{250}, no. 7, 3150--3195 (2011).

\bibitem[BB1]{MR1949165}
Bourgain, J., Brezis, H.:
On the equation {${\rm div}\, Y=f$} and application to control of phases.
J. Amer. Math. Soc, \textbf{16}, no. 2, 393--426 (2003).

\bibitem[BB2]{MR2293957}
Bourgain, J., Brezis, H.:
New estimates for elliptic equations and {H}odge type systems.
J. Eur. Math. Soc., \textbf{9}, no. 2, 277--315 (2007).

\bibitem[BvS]{MR2332419}
Brezis, H., van Schaftingen, J.:
Boundary estimates for elliptic systems with {$L^1$}-data.
Calc. Var. Partial Differential Equations, \textbf{30}, no. 3, 369--388 (2007).

\bibitem[CvS]{MR2511628}
Chanillo, S., van Schaftingen, J.:
Subelliptic {B}ourgain-{B}rezis estimates on groups.
Math. Res. Lett., \textbf{16}, no. 3, 487--501 (2009).

\bibitem[FS]{MR657581}
Folland, G. B., Stein, E. M.:
Hardy spaces on homogeneous groups.
Princeton University Press, 1982.

\bibitem[LS]{MR2122730}
Lanzani, L., Stein, E. M.:
A note on div curl inequalities.
Math. Res. Lett., \textbf{12}, no. 1, 57--61 (2005).

\bibitem[Ma]{MR2578609}
Maz'ya, V:
Estimates for differential operators of vector analysis involving {$L^1$}-norm.
J. Eur. Math. Soc., \textbf{12}, no. 1, 221--240 (2010).

\bibitem[Mi]{MR2645163}
Mironescu, P:
On some inequalities of {B}ourgain, {B}rezis, {M}az'ya, and {S}haposhnikova related to {$L^1$} vector fields.
C. R. Math. Acad. Sci. Paris, \textbf{348}, no. 9-10, 513--515 (2010).

\bibitem[MM]{MR2470831}
Mitrea, I., Mitrea, M.:
A remark on the regularity of the div-curl system.
Proc. Amer. Math. Soc., \textbf{137}, no. 5, 1729--1733 (2009).

\bibitem[S]{MR1232192}
Stein, E. M.:
Harmonic analysis: real-variable methods, orthogonality, and oscillatory integrals.
Princeton University Press, 1993.

\bibitem[vS1]{MR2078071}
van Schaftingen, J:
Estimates for {$L\sp 1$}-vector fields.
C. R. Math. Acad. Sci. Paris, \textbf{339}, no. 3, 181--186 (2004).

\bibitem[vS2]{MR2443922}
van Schaftingen, J:
Estimates for {$L^1$} vector fields under higher-order differential conditions.
J. Eur. Math. Soc., \textbf{10}, no. 4, 867--882 (2008).

\bibitem[vS3]{MR2550188}
van Schaftingen, J:
Limiting fractional and {L}orentz space estimates of differential forms.
Proc. Amer. Math. Soc., \textbf{138}, no. 1, 235--240 (2010).

\bibitem[Y]{MR2592736}
Yung, P.-L.:
Sobolev inequalities for {$(0,q)$} forms on {CR} manifolds of finite type.
Math. Res. Lett., \textbf{17}, no. 1, 177--196 (2010).

\end{thebibliography}
\end{document}